\documentclass[11pt]{amsart}

\usepackage{amsmath, amsthm, amssymb}
\input xy
\xyoption{all}
\usepackage{mathrsfs}
\usepackage{enumerate}

\newtheorem{theorem}[subsubsection]{Theorem}
\newtheorem{proposition}[subsubsection]{Proposition}
\newtheorem{lemma}[subsubsection]{Lemma}
\newtheorem{corollary}[subsubsection]{Corollary}

\theoremstyle{definition}
\newtheorem{definition}[subsubsection]{Definition}

\newtheorem{remark}[subsubsection]{Remark}

\newcommand{\Z}{\ensuremath{\mathbb{Z}}}
\newcommand{\Q}{\ensuremath{\mathbb{Q}}}
\newcommand{\C}{\ensuremath{\mathbb{C}}}

\renewcommand{\P}{\ensuremath{\mathbb{P}}}

\newcommand{\id}{\ensuremath{\textrm{id}}}
\newcommand{\M}{\ensuremath{\overline{\mathcal{M}}}}
\renewcommand{\O}{\ensuremath{\mathcal{O}}}

\newcommand{\ev}{\ensuremath{\textrm{ev}}}
\newcommand{\vir}{\ensuremath{\textrm{vir}}}
\renewcommand{\t}{\ensuremath{\mathbf{t}}}
\newcommand{\psibar}{\ensuremath{\overline{\psi}}}
\newcommand{\ch}{\ensuremath{\textrm{ch}}}
\newcommand{\Mt}{\ensuremath{\widetilde{\mathcal{M}}}}
\newcommand{\sfrac}[2]{\ensuremath{\textstyle\frac{#1}{#2}}}
\newcommand{\lfrac}[2]{\ensuremath{\displaystyle\frac{#1}{#2}}}
\renewcommand{\d}{\ensuremath{\partial}}
\newcommand{\Td}{\ensuremath{\textrm{Td}}}
\newcommand{\D}{\ensuremath{\mathcal{D}}}
\newcommand{\PP}{\ensuremath{\mathcal{P}}}

\begin{document}

\title[LG/CY correspondence for $X_{3,3}$ and $X_{2,2,2,2}$]{Landau-Ginzburg/Calabi-Yau correspondence for the complete intersections $X_{3,3}$ and $X_{2,2,2,2}$}
\author{Emily Clader}
\date{11 April 2013.}
\thanks{Partially supported by NSF RTG grant DMS-0602191}

\begin{abstract}
We define a generalization of Fan-Jarvis-Ruan-Witten theory, a ``hybrid" model associated to a collection of quasihomogeneous polynomials of the same weights and degree, which is expected to match the Gromov-Witten theory of the Calabi-Yau complete intersection cut out by the polynomials.  In genus zero, we prove that the correspondence holds for any such complete intersection of dimension three in ordinary, rather than weighted, projective space.  These results generalize those of Chiodo-Ruan for the quintic threefold, and as in that setting, Givental's quantization can be used to yield a conjectural relation between the full higher-genus theories.
\end{abstract}

\maketitle

\tableofcontents

\section{Introduction}

In the early 1990s, when the mathematical study of mirror symmetry was just beginning, physicists posited the existence of a Landau-Ginzburg/Calabi-Yau (LG/CY) correspondence connecting the geometry of Calabi-Yau complete intersections in projective space to the Landau-Ginzburg model, in which the polynomials defining the complete intersections are studied as singularities instead \cite{VW89} \cite{Wi93b}.  Mathematically, the theory on the CY side is the Gromov-Witten theory of the complete intersection, but it was not until 2007 with the series of papers \cite{FJR3}, \cite{FJR}, \cite{FJR2} that a candidate theory on the LG side was suggested, namely Fan-Jarvis-Ruan-Witten (FJRW) theory.  In \cite{CR}, the Gromov-Witten theory of the quintic threefold was shown to match the FJRW theory of the corresponding singularity in genus zero.

The goal of this paper is to extend the results of \cite{CR} to certain complete intersections in projective space.  In order to accomplish this, it is necessary to generalize FJRW theory, constructing a mathematical Landau-Ginzburg model associated to a collection of singularities rather than just one.  The theory we construct is a ``hybrid" model that combines aspects of FJRW theory and Gromov-Witten theory.

The idea for the hybrid model, as well as the technical tools required for its development, were already known by a number of authors, and this paper owes a great debt to them.  The initial definition of the hybrid moduli space was suggested by A. Chiodo, who explained it to the author and proposed the project on which this work is based.  To define a virtual cycle for the theory, we use the method of cosection localization, which is due to Kiem-Li-Chang \cite{CL} \cite{Li}.  Our application of the cosection method closely follows the work of Chang-Li \cite{CL}; in fact, for the case of the quintic threefold, the construction considered in this paper is nothing but the Landau-Ginzburg analogue of their argument.  The first application of cosection localization to the Landau-Ginzburg side is due to Chang-Li-Li, who use it in the recent paper \cite{CLL} to give an algebraic construction of FJRW theory in the case of narrow sectors.  This paper can be considered a generalization of their results.

Via Givental's quantization machinery, the genus-zero LG/CY correspondence yields a conjectural relationship between the hybrid model and the Gromov-Witten theory of the complete intersection in higher genus.  While computations of Gromov-Witten theory past genus $1$ are currently beyond the scope of mathematicians' methods, the Landau-Ginzburg model is generally thought to be more computationally manageable \cite{CR}.  Thus, if the higher-genus correspondence could be verified, it would potentially open exciting avenues for Gromov-Witten theory.

\subsection{Main result}

Given a nondegenerate collection of quasihomogeneous polynomials $W_1, \ldots, W_r \in \C[x_1, \ldots, x_N]$, each with weights $c_1, \ldots, c_N$ and degree $d$ satisfying the Calabi-Yau condition
\begin{equation}
\label{CY}
dr = \sum_{j=1}^N c_j,
\end{equation}
there are two associated theories.  On the Calabi-Yau side, one considers the complete intersection $X$ in weighted projective space cut out by the polynomials.  The cohomology of this complete intersection can be viewed as the state space from which insertions to Gromov-Witten invariants of $X$ are chosen; for any choice of $\varphi_1, \ldots, \varphi_n \in H_{GW} = H^*(X)$ and any $a_1, \ldots, a_n \in \Z^{\geq 0}$, there is a corresponding Gromov-Witten invariant
\[\langle \tau_{a_1}(\varphi_1), \ldots, \tau_{a_n}(\varphi_n)\rangle^{GW}_{g,n,\beta},\]
defined as an intersection number on the moduli space of stable maps to $X$.  The genus-zero invariants are encoded by a $J$-function
\[J_{GW}(\t, z) = z + \t + \sum_{n,\beta} \frac{1}{n!}\left\langle \t(\psi), \ldots, \t(\psi), \frac{\varphi_{\alpha}}{z - \psi}\right\rangle^{GW}_{0,n+1, \beta} \varphi^{\alpha},\]
where $\t(z) = t_0 + t_1z + t_2z^2 + \cdots \in H_{GW}[z]$ and $\varphi_{\alpha}$ runs over a basis for $H_{GW}$.  On the Landau-Ginzburg side, the polynomials $W_i$ are regarded as the equations for singularities in $\C^N$.  There is also a state space $H_{hyb}$, and its elements can be used as the insertions to hybrid invariants
\[\langle \tau_{a_1}(\phi_1), \ldots, \tau_{a_n}(\phi_n)\rangle^{hyb}_{g,n,\beta},\]
which are intersection numbers on a moduli space parameterizing stable maps to projective space together with a collection of line bundles on the source curve whose tensor powers satisfy equations determined by the polynomials $W_i$.  These, too, are encoded by a $J$-function in genus zero:
\[J_{hyb}(\t, z) = z + \t + \sum_{n,\beta} \frac{1}{n!}\left\langle \t(\overline{\psi}), \ldots, \t(\overline{\psi}), \frac{\phi_{\alpha}}{z - \overline{\psi}}\right\rangle^{hyb}_{0,n+1, \beta} \phi^{\alpha},\]
where $\t(z) = t_0 + t_1z + t_2z^2 + \cdots \in H_{hyb}[z]$ and $\phi_{\alpha}$ runs over a basis for $H_{hyb}$.  On either side, there is a grading on the state space, and the small $J$-function is defined by restricting to the degree-two component.

The genus-zero Landau-Ginzburg/Calabi-Yau (LG/CY) correspondence is the assertion that there is a degree-preserving isomorphism between the two state spaces and that, after certain identifications, the small $J$-functions coincide.  In this paper, we prove that the correspondence holds whenever the polynomials cut out a threefold complete intersection in ordinary, rather than weighted, projective space.  This leaves only three possibilities for the complete intersection: the quintic hypersurface $X_5 \subset \P^4$, the intersection of two cubic hypersurfaces $X_{3,3} \subset \P^5$, and the intersection of four quadrics $X_{2,2,2,2} \subset \P^7$.  The first of these is the content of \cite{CR}, so we focus on the second two.

After verifying the state space isomorphism in these special cases (Proposition \ref{sscorrespondence}), the strategy for proving that the small $J$-functions of Gromov-Witten theory and the hybrid model match is to relate each to an $I$-function.  On the Calabi-Yau side, the definition of $I_{GW}$ and its relationship to $J_{GW}$ were shown in \cite{Givental}.  The $I$-function can be defined explicity as a hypergeometric series in the variable $q = \exp(t_0^1)$, where $t_0 = \sum t_0^{\alpha} \varphi_{\alpha}$ and $\varphi_1 \in H^2(X)$.  Expanded in the variable $H \in H_{GW}$ corresponding to the hyperplane class, $I_{GW}$ assembles the solutions to a Picard-Fuchs equation.  In our two cases of interest, the Picard-Fuchs equation are
\[\left[ D_q^4 - 3^6q\left(D_q+\frac{1}{3}\right)^2\left(D_q+\frac{2}{3}\right)^2 \right] I_{GW} = 0\]
and
\[\left[D_q^4 - 2^8q \left(D_q+\frac{1}{2}\right)^4 \right] I_{GW}= 0,\]
for the cubic and quadric complete intersections, respectively, where $D_q = q \frac{\d}{\d q}$.  There is a ``mirror map"-- that is, an explicit change of variables
\[q' = \frac{g_{GW}(q)}{f_{GW}(q)}\]
for $\C$-valued functions $g_{GW}$ and $f_{GW}$-- under which the small $J$-function $J_{GW}$ matches $I_{GW}$:
\[\frac{I_{GW}(q,z)}{f_{GW}(q)} = J_{GW}(q',z).\]

We provide an analogous story on the Landau-Ginzburg side for each of the examples mentioned above.  Using the machinery of twisted invariants developed in \cite{Coates}, we construct a hybrid $I$-function in each case.  These are:

\begin{equation}
\label{I1}
I_{hyb}(t, z) = \sum_{\substack{d \geq 0\\ d \not \equiv -1 \mod 3}} \frac{ze^{(d +1+ \frac{H^{(d+1)}}{z})t}}{3^{6\lfloor \frac{d}{3}\rfloor}} \; \frac{ \displaystyle\prod_{\substack{1 \leq b \leq d\\ b \equiv d+1 \mod 3}} (H^{(d+1)}+bz)^{4}}{ \displaystyle\prod_{\substack{1 \leq b \leq d\\ b \not \equiv d+1 \mod 3}} (H^{(d+1)}+bz)^{2}}
\end{equation}
for the cubic and
\begin{equation}
\label{I2}
I_{hyb}(t, z) = \sum_{\substack{d \geq 0\\ d \not \equiv -1 \mod 2}} \frac{ze^{(d +1+ \frac{H^{(d+1)}}{z})t}}{2^{8\lfloor \frac{d}{2}\rfloor}} \; \frac{ \displaystyle\prod_{\substack{1 \leq b \leq d\\ b \equiv d+1 \mod 2}} (H^{(d+1)}+bz)^{4}}{ \displaystyle\prod_{\substack{1 \leq b \leq d\\ b \not \equiv d+1 \mod 2}} (H^{(d+1)}+bz)^{4}}
\end{equation}
for the quadric, where $t = t + 0z + 0z^2 + \cdots$ lies in the degree-$2$ part of the Landau-Ginzburg state space.  The key fact about these $I$-functions is that the family $I_{hyb}(t, -z)$ lies on the Lagrangian cone $\mathcal{L}_{hyb}$ on which the $J$-function is a slice, as we prove in Theorem \ref{Ifunction}.  This cone has a special geometric property that allows any function lying on it to be determined from only the first two coefficients in its expansion in powers of $z$.  Using the expressions (\ref{I1}) or (\ref{I2}), one can write
\[I_{hyb}(t, z) = \omega_1^{hyb}(t) \cdot 1^{(1)} \cdot z + \omega_2^{hyb}(t)  + O(z^{-1})\]
for explicit $\C$-valued functions $\omega_1^{hyb}(t)$ and $\omega_2^{hyb}(t)$.  We therefore obtain the following theorem:

\begin{theorem}
\label{LGCY}
Consider the hybrid model $I$-function (\ref{I1}) associated to a generic collection of two homogeneous cubic polynomials in six variables, whose coefficients when expanded in powers of $H^{(i)}$ span the solution space of the Picard-Fuchs equation 
\[\left[ D_{\psi}^4 - 3^6\psi^{-1}\left(D_{\psi}-\frac{1}{3}\right)^2\left(D_{\psi}-\frac{2}{3}\right)^2 \right] I_{hyb} = 0\]
for $D_{\psi} = \psi \frac{\d}{\d \psi}$ and $\psi = e^{3t}$.  This $I$-function and the hybrid $J$-function $J_{hyb}$ associated to the same collection of polynomials are related by an explicit change of variables (mirror map)
\[\frac{I_{hyb}(t,-z)}{\omega_1^{hyb}(t)} = J_{hyb}(t', -z), \;\;\; \text{where} \;\; t' = \frac{\omega_2^{hyb}(t)}{\omega_1^{hyb}(t)}.\]

The analogous statement holds for the hybrid model $I$-function (\ref{I2}) associated to a generic collection of four homogeneous quadric polynomials in eight variables, for which the coefficients span the solution space of the Picard-Fuchs equation
\[\left[D_{\psi}^4 - 2^8\psi^{-1} \left(D_{\psi}-\frac{1}{2}\right)^4 \right] I_{hyb}= 0\]
with $D_{\psi}  = \psi \frac{\d}{\d \psi}$ and $\psi = e^{2t}$.
\end{theorem}

The fact that the hybrid $I$-functions assemble the solutions to the specified Picard-Fuchs equations is an easy consequence of the explicit expressions for these functions.  These equations are the same as the Picard-Fuchs equations for the corresponding Calabi-Yau complete intersections after setting $q=\psi^{-1}$.  It follows that, if we use the state space correspondence to identify the state spaces $H_{GW}$ and $H_{hyb}$ in which the $I$-functions take values, then $I_{hyb}$ and the analytic continuation of $I_{GW}$ to the $\psi$-coordinate patch are both comprised of bases of solutions to the same differential equation, and hence are related by a symplectic isomorphism performing the change of basis.  This establishes the following corollary.

\begin{corollary}
\label{LGCYcor}
In either of the two cases outlined above, there is a $\C[z,z^{-1}]$-valued degree-preserving linear transformation mapping $I_{hyb}$ to the analytic continuation of $I_{GW}$ near $t=0$.  That is, the genus-zero Landau-Ginzburg/Calabi-Yau correspondence holds in these cases.
\end{corollary}

\subsection{Organization of the paper}

We begin, in Section \ref{background}, by establishing some background on quashimogeneous singularities and orbifold curves.  Section \ref{LGstatespace} is devoted to defining the state space for the Landau-Ginzburg model and proving in the two cases of interest that it is isomorphic to the state space on the Calabi-Yau side.  In Section \ref{quantumtheory}, the quantum theory of the Landau-Ginzburg model is developed for arbitrary complete intersections of the same weights and degree in weighted projective space.  At the end of that section, we specialize to the two examples of interest, and in Section \ref{correspondence}, we place those two examples in the context of Givental's quantization formalism, proving that the Lagrangian cone encoding the hybrid theory can be obtained from the Lagrangian cone encoding the genus-zero Gromov-Witten theory of projective space.  This leads to the definition of the $I$-function and the proof of the LG/CY correspondence for these two examples.

\subsection{Acknowledgments}

The author would like to thank A. Chiodo for initially suggesting the problem, explaining the basics of the hybrid model, and giving many invaluable suggestions on earlier drafts.  Y. Ruan also participated crucially in the development of the paper, by way of countless conversations, support, and advice.  The state space correspondence considered here is a very special case of an upcoming result of A. Chiodo and J. Nagel, whose argument inspired this one.  The idea of using Kiem-Li's cosection technique to define the virtual cycle in the Landau-Ginzburg model is due to H-L. Chang, J. Li, and W-P. Li, and was first brought to the attention of the author at a talk by J. Li at the 2011 Summer School on Moduli of Curves and Gromov-Witten Theory at the Institut Fourier.  Special thanks are due to J. Li for teaching the author the cosection technique, and to H-L. Chang for correspondence that clarified this topic further.

\section{Quashihomogeneous singularities and orbifold curves}
\label{background}

\subsection{Quasihomogeneous singularities}

The type of singularities for which the hybrid theory is defined are as follows.

\begin{definition}
\label{qhom}
A polynomial $W \in \C[x_1, \ldots, x_N]$ is {\it quasihomogeneous} if
\begin{enumerate}
\item there exist positive integers $c_1, \ldots, c_N$ (known as {\it weights}) and $d$ (the {\it degree}) such that
\[W(\lambda^{c_1}x_1, \ldots, \lambda^{c_N}x_N) = \lambda^d W(x_1, \ldots, x_N)\]
for all $\lambda \in \C$ and $(x_1, \ldots, x_N) \in \C^N$;
\item the {\it charges} $q_i := c_i/d$ are uniquely determined by $W$.
\end{enumerate}
\end{definition}

Let $W_1(x_1, \ldots, x_N), \ldots, W_r(x_1, \ldots, x_N)$ be a collection of quasihomogeneous polynomials in $N$ complex variables all having the same weights and degree.  Such a collection is {\it nondegenerate} if the only $\mathbf{x} \in \C^N$ for which all of the polynomials $W_i$ and all of their partial derivatives vanish is $\mathbf{x} = 0$.  It satisfies the {\it Calabi-Yau condition} if equation (\ref{CY}) holds.  All of the collections of quasihomogeneous polynomials considered in this paper will be assumed nondegenerate and Calabi-Yau.

Associated to such a collection is a group of symmetries.  In order to define this group, we will prefer to think of the $W_i$ as together defining a polynomial
\[\overline{W}(\mathbf{x}, \mathbf{p}) = p_1W_1(\mathbf{x}) + \cdots + p_r W_r(\mathbf{x}) \in \C[x_1, \ldots, x_N, p_1, \ldots, p_r].\]
From this perspective, symmetries of the collection of polynomials are simply symmetries of $\overline{W}$ in the sense of FJRW theory.  Explicitly:

\begin{definition}
The {\it group $G_{W_1, \ldots, W_r}$ of diagonal symmetries} of a collection of  quasihomogeneous polynomials of charges $c_1, \ldots, c_N$ and degree $d$ is
\[G_{W_1, \ldots, W_r} = \{(\alpha, \beta) \in (\C^*)^N \times (\C^*)^r \; | \hspace{4cm}\]
\[\hspace{3.5cm}\overline{W}(\alpha \mathbf{x}, \beta \mathbf{p}) = \overline{W}(\mathbf{x}, \mathbf{p}) \text{ for all } (\mathbf{x}, \mathbf{p}) \in \C^N \times \C^r\}.\]
\end{definition}

The group of diagonal symmetries always contains the subgroup
\[J = \{(t^{c_1}, \ldots, t^{c_N}, t^{-d}, \ldots, t^{-d}) \; | \; t \in \C^*\}.\]
This is the analogue of the group denoted $\langle J \rangle$ in FJRW theory.  There is an extra datum in the definition of FJRW theory that will not be present in the current paper: a subgroup $G$ of the group of diagonal symmetries containing $J$.  The theory developed here corresponds to the choice $G= J$.

\subsection{Orbifold curves and orbifold stable maps}

Similarly to FJRW theory, the hybrid model concerns curves equipped with a collection of line bundles whose tensor powers satisfy certain conditions.  Given that the moduli problem of roots of line bundles is better-behaved with respect to orbifold curves (see, for example, Section 1.2 of \cite{Chiodo}), the underlying curves of the theory should be allowed limited orbifold structure.

\begin{definition}
An {\it orbifold curve} (or ``balanced twisted curve" \cite{AV}) is a one-dimensional Deligne-Mumford stack with a finite ordered collection of marked points and at worst nodal singularities such that
\begin{enumerate}
\item the only points with nontrivial stabilizers are marked points and nodes;
\item all nodes are {\it balanced}; i.e., in the local picture $\{xy = 0\}$ at a node, the action of the isotropy group $\Z_k$ is given by
\[(x,y) \mapsto (\zeta_k x, \zeta_k^{-1}y)\]
with $\zeta_k$ a primitive $k$th root of unity.
\end{enumerate}
\end{definition}

Throughout the paper, we will denote the coarse underlying space of an orbifold curve $C$ by $|C|$.

\subsubsection{Multiplicities of orbifold line bundles}

Let $C$ be an orbifold curve and let $L$ an orbifold line bundle on $C$.  Choose a node $n$ of $C$ with isotropy group $\Z_{\ell}$ and a distinguished branch of $n$, so that the local picture can be expressed as $\{xy=0\}$ with $x$ being the coordinate on the distinguished branch.  Let $g$ be a generator of the isotropy group $\Z_{\ell}$ at the node acting on these local coordinates by $g\cdot (x,y) = (\zeta_{\ell}x, \zeta_{\ell}^{-1}y)$.

\begin{definition}
\label{multiplicity}
The {\it multiplicity} of $L$ at (the distinguished branch of) the node $n$ is the integer $m \in \{0, \ldots, \ell - 1\}$ such that, in local coordinates $(x, y, \lambda)$ on the total space of $L$, the action of $g$ is given by
\[g\cdot(x,y,\lambda) = (\zeta_{\ell}x, \zeta_{\ell}^{-1}y, \zeta_{\ell}^m \lambda).\]
In the same way, one can define the multiplicity of $L$ at a marked point by the action of a generator of the isotropy group on the fiber.
\end{definition}

One extremely important property of the multiplicity is that it allows one to determine the equation satisfied by the coarsening of $L$ on each of its components \cite{Chiodo} \cite{CR}.  Suppose that
\[\nu: L^{\otimes \ell} \rightarrow N\]
is an isomorphism between a power of $L$ and a line bundle $N$ pulled back from the coarse curve $|C|$ and $Z \subset C$ is a smooth (that is, non-nodal) irreducible component of $C$.  Let $m_1, \ldots, m_k$ be the multiplicities of $L$ at the nodes where $Z$ meets the rest of $C$, where in each case the distinguished branch is the one lying on $Z$.  Let $\epsilon: C \rightarrow |C|$ be the coarsening map.  If $|L| = \epsilon_* N$, then we have an isomorphism
\begin{equation}
\label{reification}
\epsilon_*\nu: |L|^{\otimes \ell} \rightarrow N \otimes \O_{|Z|}\left(-\sum_{i=1}^k m_i[p_i]\right),
\end{equation}
where $p_1, \ldots, p_k$ are the images in $|Z|$ of the points where $Z$ meets the rest of $C$.

Since $\epsilon$ is flat, $|L|$ is a line bundle; in particular, the fact that it has integral degree can often be used to find constraints on the multiplicities of $L$.  Conversely, the multiplicities at all of the marked points and nodes of $C$, together with the bundle $|L|$ on $|C|$, collectively determine $L$ as an orbifold line bundle.  See Lemma 2.2.5 of \cite{CR} for a precise statement to this effect.

\section{Landau-Ginzburg state space}
\label{LGstatespace}

Let $W_1, \ldots, W_r \in \C[x_1, \ldots, x_N]$ be a collection of quasihomogeneous polynomials of the same weights and degree.  In this section, we define the state space associated to such a collection and prove that in the cases of interest for this paper, there is a degree-preserving isomorphism between the Landau-Ginzburg state space and the state space for the Gromov-Witten theory of the corresponding complete intersection.

\subsection{State space}

The {\it state space} of the hybrid theory is the following vector space:
\begin{equation}
\label{ssdef}
\mathcal{H}_{hyb}(W_1, \ldots, W_r) = H^*_{CR}\left( \frac{\C^N \times (\C^r \setminus \{0\})}{J}, \overline{W}^{+ \infty}; \C \right),
\end{equation}
where $\overline{W}^{+ \infty} = (\text{Re} \overline{W})^{-1}( ( \rho, + \infty) )$ for $\rho \gg 0$ and $J$ acts by multiplication in each factor.

As a vector space, Chen-Ruan cohomology is the cohomology of the inertia stack, whose objects are pairs $((\mathbf{x}, \mathbf{p}), \gamma)$, where $\gamma \in J$, $(\mathbf{x}, \mathbf{p}) \in \C^N \times (\C^r \setminus \{0\})$, and $\gamma (\mathbf{x}, \mathbf{p}) = (\mathbf{x}, \mathbf{p})$.  The only elements of $J$ with nontrivial fixed-point sets are those of the form
\[(t^{c_1}, \ldots, t^{c_N}, 1, \ldots, 1),\]
where $t$ is a $d$th root of unity, so such elements index the components of the inertia stack.  These components are known as {\it twisted sectors}.

\subsubsection{Degree shifting}

As is usual in Chen-Ruan cohomology, we should shift the degree.

\begin{definition}
\label{degreeshift}
Let $\gamma = (e^{2\pi i \Theta_1^{\gamma}}, \ldots, e^{2\pi i \Theta_N^{\gamma}}, 1, \ldots, 1) \in J$ be an element with nontrivial fixed-point set, where $\Theta_i^{\gamma} \in \{0, \frac{1}{d}, \ldots, \frac{d-1}{d}\}$.  The {\it degree-shifting number} or {\it age shift} for $\gamma$ is
\[\iota(\gamma) = \sum_{j=1}^N (\Theta_j^{\gamma} - q_j),\]
where $q_j$ are the charges defined in Definition \ref{qhom}.
\end{definition}

Now, given $\alpha \in \mathcal{H}_{hyb}(W_1, \ldots, W_r)$ from the twisted sector indexed by $\gamma$, we set
\[\deg_{\overline{W}}(\alpha) = \deg(\alpha) + 2\iota(\gamma),\]
where $\deg(\alpha)$ denotes the ordinary degree of $\alpha$ as an element of the cohomology of the inertia stack.  This gives a grading on $\mathcal{H}_{hyb}(W_1, \ldots, W_r)$.

\subsubsection{Broad and narrow sectors}
\label{broadnarrow}

A twisted sector indexed by an element $(t^{c_1}, \ldots, t^{c_N}, 1, \ldots, 1) \in J$ will be called {\it narrow} if there is no $i$ with $t^{c_i} = 1$.  This condition ensures that the sector is supported on the suborbifold
\[\frac{\{0\} \times (\C^r \setminus \{0\})}{J} \subset \frac{\C^N \times (\C^r \setminus \{0\})}{J},\]
whose coarse underlying space is $\P^{r-1}$.  Since the above is disjoint from $\overline{W}^{+ \infty}$, the relative cohomology on these sectors is an absolute cohomology group, and indeed, each narrow sector is isomorphic to $H^*(\P^{r-1})$.  A sector that is not narrow will be called {\it broad}.

\subsubsection{Cases of interest}
\label{interest}

For most of the paper, we will restrict to the cases mentioned in the introduction, in which $W_1, \ldots, W_r$ define a threefold Calabi-Yau complete intersection in ordinary, rather than weighted, projective space.  This leaves the following three possibilities:
\begin{enumerate}
\item $r=1, d = 5, N = 5$ (quintic hypersurface in $\P^4$);
\item $r=2, d = 3, N=6$ (intersection of two cubics in $\P^5$);
\item $r=4, d = 2, N=8$ (intersection of four quadrics in $\P^7$).
\end{enumerate}
The first case was handled in \cite{CR}, while the second and third are considered here.

In case (2), the state space is
\[H^*_{CR}\left( \frac{\C^6 \times (\C^3\setminus \{0\})}{\C^*}, \overline{W}^{+ \infty}; \C\right),\]
where $\C^*$ acts via
\begin{equation}
\label{action}
\lambda(x_1, \ldots, x_6, p_1, p_2, p_3) = (\lambda, \ldots, \lambda, \lambda^{-3}, \lambda^{-3}, \lambda^{-3}).
\end{equation}
The orbifold in question, then, is the total space of the orbifold vector bundle $\O_{\P(3,3)}\left(-1\right)^{\oplus 6}$, where
\[\O_{\P(3,3)}\left(-1\right) =  \frac{(\C^3 \setminus \{0\}) \times \C}{\C^*}\]
with $\C^*$ acting with weights $(3,3,3, -1)$.\footnote{If one considers $\P(3,3)$ as arising via the root construction applied to $\P^1$ with its line bundle $\O(-1)$, this is the natural third root of the pullback of $\O(-1)$ (see Section 2.1.5 of \cite{Johnson}).}  The only broad sector is the nontwisted sector, while the twisted (narrow) sectors each contribute $H^*(\P^1)$.  Thus, the decomposition of the state space into sectors is:
\[H^*(\O_{\P^1}(-1)^{\oplus 6}, \overline{W}^{+ \infty}) \oplus H^*(\P^1) \oplus H^*(\P^1).\]
A similar analysis shows that the state space in case (3) is
\[H^*(\O_{\P^3}(-1)^{\oplus 8}, \overline{W}^{+ \infty}) \oplus H^*(\P^3).\]

\subsection{Cohomological LG/CY correspondence}

In the two new cases mentioned above, we verify that the state space isomorphism, or cohomological LG/CY correspondence, holds.  This is only a simple special case of a general state space isomorphism for Calabi-Yau complete intersections that will be proved in upcoming work of Chiodo and Nagel \cite{CN}, and which was discussed in a talk by J. Nagel at the Workshop on Recent Developments on Orbifolds at the Chern Institute of Mathematics in July 2011.

\begin{proposition}
\label{sscorrespondence}
Let $W_1(x_1, \ldots, x_6)$ and $W_2(x_1, \ldots, x_6)$ be homogeneous cubic polynomials defining a complete intersection $X_{3,3} \subset \P^6$.  Then the hybrid state space associated to these polynomials is isomorphic to the Gromov-Witten state space of $X_{3,3}$; that is,
\begin{equation}
\label{cubicss}
H^*(\O_{\P^1}(-1)^{\oplus 6}, \overline{W}^{+ \infty}) \oplus H^*(\P^1) \oplus H^*(\P^1) \cong H^*(X_{3,3}).
\end{equation}
Moreover, this isomorphism is degree-preserving under the degree shift (\ref{degreeshift}) for the left-hand side.

Similarly, there is a degree-preserving state space isomorphism for a collection of eight quadrics defining a complete intersection $X_{2,2,2,2} \subset \P^7$:
\[H^*(\O_{\P^3}(-1)^{\oplus 8}, \overline{W}^{+ \infty}) \oplus H^*(\P^3) \cong H^*(X_{2,2,2,2}).\]
\begin{proof}
The three summands on the left-hand side of (\ref{cubicss}) have degree shifts $-2, 0,$ and $2$, respectively.  Thus, the narrow sectors contribute one-dimensional summands in degrees $0$, $2$, $4$, and $6$.  By the Lefschetz hyperplane principle, this matches the primitive cohomology of $X_{3,3}$, so all that remains in the cubic case is to prove that
\[H^k(\O_{\P^1}(-1)^{\oplus 6}, \overline{W}^{+ \infty}) \cong \begin{cases} H^3(X_{3,3}) & k=7\\ 0 & \text{otherwise}.\end{cases}\]
Similarly, in the quadric case, the only statement that is not immediate is
\[H^k(\O_{\P^3}(-1)^{\oplus 8}, \overline{W}^{+ \infty}) \cong \begin{cases} H^3(X_{2,2,2,2}) & k=11\\0 & \text{otherwise}.\end{cases}\]
Both arguments are elementary, so we describe only the cubic case.

The basic idea is to relate both state spaces to the quotient
\begin{equation}
\label{quotient}
\frac{\C^6 \times \C^2}{\C^*},
\end{equation}
where $\C^*$ acts with weights $(1,\ldots, 1, -3,-3)$\footnote{In fact, this picture is extremely useful for understanding the motivation behind the LG/CY correspondence more generally.  Even the moduli spaces on each side can be viewed as arising from such a dichotomy; see Remark \ref{dichotomy}.}.  The subquotient where the $\C^6$ coordinate is nonzero is $\O_{\P^5}(-3)^{\oplus 2}$, and the Thom isomorphism implies that
\[H^3(X_{\overline{W}}) \cong H^7(\O_{\P^5}(-3)^{\oplus 2}, \overline{W}^{+ \infty}).\]
Letting $A$ denote the complement of the zero section in this bundle, the long exact sequence of the triple $(\O_{\P^5}(-3)^{\oplus 2}, A, \overline{W}^{+ \infty})$ implies further that
\[H^3(X_{\overline{W}}) \cong H^7(A, \overline{W}^{+ \infty}).\]
On the other hand, the subquotient of (\ref{quotient}) on which the $\C^2$ coordinate is nonzero is $\O_{\P(3,3)}(-1)^{\oplus 6}$, which also contains $A$ as the complement of the zero section.  The long exact sequence of the triple $(\O_{\P(3,3)}(-1)^{\oplus 6}, A, \overline{W}^{+ \infty})$ shows that
\[H^7(\O_{\P(3,3)}(-1)^{\oplus 6}, \overline{W}^{+ \infty}) \cong H^7(A, \overline{W}^{+ \infty}),\]
as required.
\end{proof}
\end{proposition}

\section{Quantum theory for the Landau-Ginzburg model}
\label{quantumtheory}

\subsection{Moduli space}

Let $W_1, \ldots, W_r$ be a nondegenerate collection of quasihomogeneous polynomials, each having weights $c_1, \ldots, c_N$ and degree $d$.  Let $\overline{d}$ denote the {\it exponent} of the group $G_{W_1, \ldots, W_r}$; i.e., the smallest integer $k$ for which $g^k=1$ for all $g \in G_{W_1, \ldots, W_r}$.\footnote{In the examples of interest in this paper, we will have $\overline{d} = d$, but this is not necessarily the case in general.}  For each $i$, set
\[\overline{c_i} = c_i \frac{\overline{d}}{d},\]
where as usual $d$ is the degree of the polynomials $W_i$, and $c_i$ are the weights.

\begin{definition}
A genus-$g$, degree $\beta$ Landau-Ginzburg stable map with $n$ marked points over a base $T$ is given by the following objects:
\[\xymatrix{
\mathscr{L}\ar[r] & (\mathscr{C}, \{\mathscr{S}_i\})\ar[r]^f\ar[d]^{\pi} & \P^{r-1}\\
& T, & 
}\]
together with an isomorphism
\[\varphi: \mathscr{L}^{\otimes d} \xrightarrow{\sim} \omega_{\log} \otimes f^*\O(-1),\]
where
\begin{enumerate}
\item $\mathscr{C}/T$ is a genus-$g$, $n$-pointed orbifold curve;
\item For $i=1, \ldots, n$, the substack $\mathscr{S}_i \subset \mathscr{C}$ is a (trivial) gerbe over $T$ with a section $\sigma_i: T \rightarrow \mathscr{S}_i$ inducing an isomorphism between $T$ and the coarse moduli of $\mathscr{S}_i$;
\item $f$ is a morphism whose induced map between coarse moduli spaces is an $n$-pointed genus $g$ stable map of degree $\beta$;
\item $\mathscr{L}$ is an orbifold line bundle on $\mathscr{C}$ and $\varphi$ is an isomorphism of line bundles;
\item For any $p \in \mathscr{C}$, the representation $r_p: G_p \rightarrow \Z_d$ given by the action of the isotropy group on the fiber of $\mathscr{L}$ is faithful.
\end{enumerate}
\end{definition}

\begin{definition}
A morphism between two Landau-Ginzburg stable maps $(\mathscr{C}/T, \{\mathscr{S}_i\}, f, \mathscr{L}, \phi)$ and $(\mathscr{C}'/T', \{\mathscr{S}'_i\}, f', \mathscr{L}', \phi')$ is a tuple of morphisms $(\tau, \mu, \alpha)$, where $(\tau, \mu)$ forms a morphism of pointed orbifold stable maps:
\[\xymatrix{
T\ar[dd]^{\tau} &\mathscr{C}\ar[dd]^{\mu} \ar[l]\ar[dr] & \\
 & & \P^{r-1}\\
 T' & \mathscr{C}' \ar[l]\ar[ur] &
 }\]
 and $\alpha: \mu^*\mathscr{L}' \rightarrow \mathscr{L}$ is an isomorphism of line bundles such that
\[\phi \circ \alpha^{\otimes d} = \delta \circ \mu^*\phi',\]
where $\delta: \mu^*\omega_{\mathscr{C}'} \rightarrow \omega_{\mathscr{C}}$ is the natural map.
\end{definition}

\begin{definition}
The {\it hybrid model moduli space} is the stack $\Mt_{g,n}^d(\P^{r-1}, \beta)$ parameterizing $n$-pointed genus-$g$ Landau-Ginzburg stable maps of degree $\beta$, up to isomorphism.
\end{definition}

Before we prove that this is a proper Deligne-Mumford stack, a few remarks are in order.

\begin{remark}
Landau-Ginzburg stable maps can be viewed as tensor products of stable maps to $\P(d, \ldots, d)$ and spin structures.  Indeed, the datum of a stable map to $\P(d, \ldots, d)$ is equivalent to a map $f: C \rightarrow \P^{r-1}$ together with a $d$th root of the line bundle $f^*\O(1)$, while a spin structure on $C$ is a $d$th root of $\omega_{\log}$.
\end{remark}

\begin{remark}
It would in some sense be more natural to define Landau-Ginzburg stable maps as maps to a weighted projective space $\P(d, \ldots, d)$ rather than the coarse underlying $\P^{r-1}$.  In fact, though, this is equivalent to what we have done, since if $f: C \rightarrow \P(d,\ldots, d)$ is an orbifold stable map and there exists a line bundle $L$ on $C$ such that $L^{\otimes d} \cong f^*\O(-1) \otimes \omega_{\log}$, then $f^*\O(-1)$ is forced to have integral degree, which implies that $f$ factors through a map to $\P^{r-1}$.
\end{remark}

\begin{remark}
\label{FJRWtheory}
In the case where $r=1$, the above is not exactly the same as the moduli space of $W$-structures in FJRW theory.  However, Proposition 2.3.13 of \cite{CR} shows that the map
\[\Mt_{g,n}^d(\P^0, 0) \rightarrow W_{g,n, \langle J \rangle}\]
\[(C,f,L,\varphi) \mapsto (C, (L^{\otimes c_1}, \varphi^{\overline{c_1}}), \ldots, (L^{\otimes c_N}, \varphi^{\overline{c_N}}))\]
is surjective and locally isomorphic to $B\mu_d \rightarrow B(\mu_{\overline{d}})^N$, so integrals over $W_{g,n,\langle J \rangle}$ can be expressed as integrals over $\Mt_{g,n,}^d(\P^0, 0)$, and the correlators defined below agree with those in FJRW theory.
\end{remark}

Forgetting the line bundle $\mathscr{L}$ and the orbifold structure gives a morphism
\[\rho: \Mt_{g,n}^d(\P^{r-1}, \beta) \rightarrow \M_{g,n}(\P^{r-1}, \beta).\]
This map is quasifinite (see Remark 2.1.20 of \cite{FJR}).  Indeed, for any orbifold stable map $f: C \rightarrow \P^{r-1}$, any two choices of $L$ such that $L^{\otimes d} \cong \omega_{\log} \otimes f^*\O(-1)$ differ by a choice of a line bundle $N$ with an isomorphism $\xi: N^{\otimes d} \cong \O_{C}$.  The set of isomorphism classes of such pairs $(N, \xi)$ is isomorphic to the finite group $H^1(C, \Z_d)$.

\begin{proposition}
For any nondegenerate collection of quasihomogeneous polynomials $\overline{W}$ as above, the stack $\Mt_{g,n}^d(\P^{r-1}, \beta)$ is a proper Deligne-Mumford stack with projective coarse moduli.
\begin{proof}
The proof follows closely that of Thereom 2.2.6 of \cite{FJR} and uses repeatedly the identification between orbifold line bundles on $\mathscr{C}$ and maps $\mathscr{C} \rightarrow B\C^*$.  Given $(C, \{\sigma_i\}, f) \in \M_{g,n}(\P^{r-1}, \beta)$, an element of $\rho^{-1}(C, \{\sigma_i\}, f)$ is given by a map
\[\mathfrak{L}: C \rightarrow B\C^*\]
such that
\[\xymatrix{
 & B\C^*\ar[d]\\
C\ar[ur]^{\mathfrak{L}}\ar[r]_{\delta} & B\C^*
 }\]
 commutes, where $\delta$ is the map corresponding to the line bundle $f^*\O(-1) \otimes \omega_{\log}$ and the vertical arrow is $x \mapsto x^d$.
 
 Let $C_{\M} \rightarrow \M_{g,n}(\P^{r-1}, \beta)$ denote the universal family, and abbreviate $\M = \M_{g,n}(\P^{r-1}, \beta)$.  Let $C_{\overline{W}}$ be the fiber product
 \[\xymatrix{
 C_{\overline{W}}\ar[d] \ar[r] & B\C^*\ar[d]\\
 C_{\M}\ar[r]^{\delta}& B\C^*,
 }\]
 with the right vertical arrow as before.  Note that $C_{\overline{W}}$ is an \'etale gerbe over $C_{\M}$ banded by $\Z_d$, so it is a Deligne-Mumford stack.
 
Any Landau-Ginzburg stable map $(\mathscr{C}/T, \{\mathscr{S}_i\}, f,\mathscr{L}, \phi)$ induces a representable morphism $\mathscr{C} \rightarrow C_{\overline{W}}$ which is a balanced twisted stable map, for which the homology class of the image of the coarse curve $C$ is the class $F$ of a fiber of the universal curve $C_{\M} \rightarrow \M$.  Furthermore, the family of coarse curves and maps $(C, \{\sigma_i\}, f) \rightarrow T$ gives rise to a morphism $T \rightarrow \M$, and we have an isomorphism $C \cong T \times_{\M} C_{\M}$.  Thus, there is a basepoint-preserving functor
 \[\Mt_{g,n}^d(\P^{r-1}, \beta) \rightarrow \mathscr{H}_{g,n}(C_{\overline{W}}/\M, F),\]
where the latter denotes the stack of balanced, $n$-pointed twisted stable maps of genus $g$ and class $F$ into $C_{\overline{W}}$ relative to the base stack $\M$ (see Section 8.3 of \cite{AV}).  The image lies in the closed substack where the markings of $C$ line up over the markings of $C_{\M}$, and the functor given by the restricting to this substack is an equivalence.  Thus, the results of \cite{AV} imply that $\Mt_{g,n}^d(\P^{r-1}, \beta)$ is a proper Deligne-Mumford stack with projective coarse moduli.
\end{proof}
\end{proposition}

\subsubsection{Decomposition by multiplicities}

With the analogy to FJRW theory mentioned in Remark \ref{FJRWtheory} in mind, one obtains a decomposition of the hybrid moduli space just as in Proposition 2.3.7 of \cite{CR}.  In $W_{g,n, \langle J \rangle}$, let $\gamma_i \in \text{Aut}(W)$ give the multiplicities of $L^{\otimes c_1}, \ldots, L^{\otimes c_N}$ at the $i$th marked point.  Then the condition that $\gamma_i \in J$ implies that there exists $e^{2\pi i \frac{m_i}{d}} \in \Z_d$ such that $e^{2\pi i \frac{m_i \overline{c_j}}{\overline{d}}} = e^{2\pi i \frac{m_{i,j}}{\overline{d}}}$ for all $j$, so $\gamma_i$ is determined by $m_i \in \{0, 1, \ldots, d-1\}$.  Let
\[\Mt^d_{g,n}(\P^{r-1}, \beta) = \bigsqcup_{m_1, \ldots, m_n \in \Z_d} \Mt^d_{g,(m_1, \ldots, m_n)}(\P^{r-1},\beta),\]
where $\Mt^d_{g,(m_1, \ldots, m_n)}(\P^{r-1}, \beta)$ is the substack in which the multiplicity of $L^{\otimes c_j}$ at the $i$th marked point is $m_{i,j} \equiv m_i \overline{c_j} \mod \overline{d}$, or equivalently, the multiplicity of $L$ at the $i$th marked point is $m_i$.  The following terminology will be used later:

\begin{definition}
\label{narrow}
A marking or node is called {\it narrow} if all of the line bundles $L^{\otimes c_1}, \ldots, L^{\otimes c_N}$ have nonzero multiplicity $m_{i,j} \in \Z_{\overline{d}}$, and is called {\it broad} otherwise.  (In the literature, these situations are sometimes referred to as {\it Neveu-Schwartz} and {\it Ramond}, respectively.)
\end{definition}

\begin{remark}
It is no accident that this terminology coincides with that used for sectors of the state space in Section \ref{broadnarrow}.  Indeed, elements of $J$ index both sectors of the state space and components of the moduli space, and the narrow sectors of the state space correspond to components of the moduli space in which every marked point is narrow.
\end{remark}

\subsection{Virtual cycle}
\label{cosection}

In order to define a virtual cycle on the hybrid moduli space, we will make use of the cosection technique developed in \cite{Li}, \cite{CL}, and \cite{CLL}.

\subsubsection{Construction of the virtual cycle}

Since the hybrid model correlators will be defined as integrals over the substacks $\Mt^d_{g,(m_1, \ldots, m_n)}(\P^{r-1}, \beta)$, it suffices to define a virtual cycle on each of these.  In fact, we will only define the virtual cycle for the narrow components-- that is, when $m_{i,j} \neq 0 \in \Z_{\overline{d}}$ for all $i$ and $j$.  This implies, in particular, that $m_i \geq 1$ for all $i$.

By passing to the coarse underlying curve, an element $(C, f, L, \varphi) \in \Mt^d_{g,(m_1, \ldots, m_n)}(\P^{r-1}, \beta)$ is equivalent to a tuple $(C, f, L, \varphi)$ in which $f: C \rightarrow \P^{r-1}$ is a non-orbifold stable map and $\varphi$ is an isomorphism
\[L^{\otimes d} \cong f^*\O(-1) \otimes \omega_{\log} \otimes \O\left(-\sum_{i=1}^n m_i [x_i]\right);\]
see (\ref{reification}) and Lemma 2.2.5 of \cite{CR}.  In what follows, we will view elements of $\Mt^d_{g, (m_1, \ldots, m_n)}(\P^{r-1}, \beta)$ from this perspective.

Consider the stack $\mathcal{P}$ parameterizing tuples $(C,f, L, \varphi, s_1, \ldots, s_N)$, in which $(C,f,L,\varphi) \in \Mt_{g,(m_1, \ldots, m_n)}^d(\P^{r-1},\beta)$ and $s_i \in H^0(C,L^{\otimes c_i})$.  This is in general not proper; it should be viewed as the Landau-Ginzburg analogue of Chang and Li's moduli space of stable maps with $p$-fields \cite{CL}.  In their paper, Chang and Li exhibit a relative perfect obstruction theory on $\mathcal{P}$ relative to the Artin stack $\mathcal{D}_{g}$ parameterizing genus-$g$ curves with a line bundle of fixed degree.  While the present situation also requires marked points, the same construction applies.

Namely, denote by $\mathcal{D}_{g,n}$ the moduli stack of genus-$g$, $n$-pointed curves with a line bundle of fixed degree.  Let $\mathscr{L}_{\D_{g,n}}$ be the universal line bundle over $\D_{g,n}$, let $\pi_{\D_{g,n}}: \mathcal{C}_{\D_{g,n}} \rightarrow \D_{g,n}$ be the universal family, and let
\[\mathscr{P}_{\D_{g,n}} = \mathscr{L}_{\D_{g,n}}^{\otimes -d} \otimes \omega_{\mathcal{C}_{\D_{g,n}}/\D_{g,n}} \otimes \O\left(\sum_{i=1}^n(1-m_i)[x_i]\right).\]
Then $\PP$ embeds into the moduli of sections of
\[\mathcal{Z} = \text{Vb}(\mathscr{L}_{\D_{g,n}}^{\oplus N} \oplus \mathscr{P}_{\D_{g,n}}^{\oplus r})\]
over $\D_{g,n}$ (see Section 2.2 of \cite{CL}), where $\text{Vb}$ denotes the total space of a vector bundle. 

Similarly, over $\PP$, let $\mathscr{L}$ be the universal line bundle, $\pi: \mathcal{C}_{\mathcal{P}} \rightarrow \mathcal{P}$ be the universal family, and $\mathscr{P} = f^*\O(1) = \mathscr{L}^{\otimes - d} \otimes \omega_{\mathcal{C}_{\PP}/\PP} \otimes \O(\sum_{i=1}^n (1-m_i)[x_i])$.  The tautological
\[\mathfrak{s}_i \in \Gamma(\mathcal{C}_{\PP}, \mathscr{L}^{\otimes c_i}) \;\;\; \text{ and } \;\;\; \mathfrak{p}_j \in \Gamma(\mathcal{C}_{\PP}, \mathscr{P}),\]
in which the latter are given by the pullbacks of coordinate sections of $\O_{\P^{r-1}}(1)$, combine to give a map $\mathcal{C}_{\PP} \rightarrow \text{Vb}\left(\bigoplus_{i=1}^N\mathscr{L}^{\otimes c_i}_{\D_{g,n}} \oplus \mathscr{P}^{\oplus r}_{\D_{g,n}}\right) \times_{\mathcal{C}_{\D_{g,n}}} \mathcal{C}_{\PP}$ which is a section of the projection map.  Composing this with the projection to the first factor yields an ``evaluation map"
\[\mathfrak{e}: \mathcal{C}_{\PP} \rightarrow \mathcal{Z}.\]
Using Proposition 2.5 of \cite{CL} and the canonical isomorphism
\[\mathfrak{e}^* \Omega^{\vee}_{\mathcal{Z}/\mathcal{C}_{\D_{g,n}}} \cong \bigoplus_{i=1}^N \mathscr{L}^{\otimes c_i} \oplus \mathscr{P}_{\D_{g,n}}^{\oplus r},\]
one finds that there is a relative perfect obstruction theory
\[\mathbb{E}_{\mathcal{P}/\mathcal{D}_{g,n}} = R^{\bullet}\pi_*(\mathscr{L}^{\otimes c_1}\oplus \cdots \oplus \mathscr{L}^{\otimes c_N} \oplus \mathscr{P}^{\oplus r}).\]

Thus, we have $\mathcal{O}b_{\mathcal{P}/\mathcal{D}_{g,n}} = R^1\pi_*(\bigoplus_{i=1}^N\mathscr{L}^{\otimes c_i} \oplus \mathscr{P}^{\oplus r})$.  The polynomial $\overline{W}$ defines a cosection-- that is, a homomorphism
\[\sigma: \mathcal{O}b_{\mathcal{P}/\mathcal{D}_{g,n}} \rightarrow \O_{\mathcal{P}}.\]
To define $\sigma$, fix an element $\xi = (C,f,L, \varphi, s_1, \ldots, s_N) \in \mathcal{P}$ and let $p_j = f^*x_j \in H^0(C, f^*\O(1))$, where $x_j \in H^0(\P^{r-1}, \O(1))$ are the coordinate functions.  Take an \'etale chart $T \rightarrow \mathcal{P}$ around $\xi$ with $\mathcal{C}_T = \mathcal{C}_{\mathcal{P}} \times_{\mathcal{P}} T$.  Then $\sigma$ is defined in these local coordinates as the map
\[H^1(\mathcal{C}_T, \mathscr{L}^{\otimes c_1} \oplus \cdots \oplus \mathscr{L}^{\otimes c_N}) \oplus H^1(\mathcal{C}_T ,\mathscr{P}^{\oplus r}) \rightarrow \C\]
given by sending $(\tilde{s}_1, \ldots, \tilde{s}_N, \tilde{p}_1, \ldots, \tilde{p}_r)$ to
\[\sum_{i=1}^N\frac{c_i}{d}\frac{\d \overline{W}}{\d x_i}(s_1, \ldots, s_N, p_1, \ldots, p_r)\cdot \tilde{s}_i - \sum_{j=1}^r \frac{\d \overline{W}}{\d p_j}(s_1, \ldots, s_N, p_1, \ldots, p_r) \cdot \tilde{p}_j.\]
The fact that this is canonically an element of $\C$ relies crucially on the fact that $m_i \geq 1$ for all $i$.  For example, $\frac{\d \overline{W}}{\d p_j}(s_1, \ldots, s_N, p_1, \ldots, p_r)$ lies in
\[H^0(\mathscr{L}^{\otimes d}) = H^0\left(\mathscr{P}^{\vee} \otimes \omega_{\mathscr{C}_{\PP}/\PP} \otimes \sum(m_i - 1)[x_i]\right) \hookrightarrow H^1(\mathscr{P})^{\vee}\]
by Serre duality.

The {\it degeneracy locus} of $\sigma$ is defined as the substack $\mathcal{D}(\sigma)$ of $\mathcal{P}$ on which the fiber of $\sigma$ is the zero homomorphism.  This is the locus on which the localized virtual cycle will be supported.

\begin{lemma}
\label{degeneracy}
The degeneracy locus of $\sigma$ is precisely $\Mt_{g,(m_1, \ldots, m_n)}^d(\P^{r-1}, \beta)$.
\begin{proof}
The hybrid moduli space $\Mt_{g,(m_1, \ldots, m_n)}^d(\P^{r-1}, \beta)$ embeds in $\mathcal{P}$ as the locus where $s_1 = \cdots = s_N = 0$, and it is clear that the fiber of $\sigma$ is identically zero on this locus.  Conversely, if $(s_1, \ldots, s_N) \neq 0$, then either $(s_1, \ldots, s_N)$ does not lie in the common vanishing locus of the polynomials $W_i$, or there is some $i$ for which not every $\frac{\d W_j}{\d x_i}(s_1, \ldots, s_N)$ vanishes.  In the first case, if $W_j(s_1, \ldots, s_N) \neq 0$, then one can choose $\tilde{p}_j$ so that
\[W_j(s_1, \ldots, s_N) \cdot\tilde{p}_j = \frac{\d \overline{W}}{\d p_j}(s_1, \ldots, s_N, p_1, \ldots, p_r) \cdot \tilde{p}_j \neq 0,\]
so taking all other $\tilde{p}_i$'s and all $\tilde{s}_i$'s to be zero shows that the fiber of $\sigma$ over $\xi$ is not identically zero.  In the second case, independence of the sections $p_j$ shows that
\[\sum_{j=1}^r p_j\frac{\d W_j}{\d x_i}(s_1, \ldots, s_N) = \frac{\d \overline{W}}{\d x_i}(s_1, \ldots, s_N, p_1, \ldots, p_r) \neq 0.\]
Thus, there exists $\tilde{s}_i$ such that $\frac{\d \overline{W}}{\d x_i}(s_1, \ldots, s_N, p_1, \ldots, p_r)\cdot \tilde{s}_i \neq 0$, so again one can choose all other $\tilde{s}_j$'s and all $\tilde{p}_j$'s to be zero to see that the fiber of $\sigma$ over $\xi$ is not identically zero.
\end{proof}
\end{lemma}

\begin{remark}
By studying $\sigma$ a bit more carefully, one notices that it descends to the obstruction theory of $\mathcal{P}$ relative to $\M_{g,n}$ rather than $\mathcal{D}_{g,n}$.\footnote{The following argument was suggested by H.-L. Chang in correspondence with Y. Ruan.}  To do so, consider the deformation exact sequence
\begin{equation}
\label{deformations}
T_{\mathcal{D}_{g,n}/\M_{g,n}} \xrightarrow{\tau} \mathcal{O}b_{\mathcal{P}/\mathcal{D}_{g,n}} \rightarrow \mathcal{O}b_{\mathcal{P}/\M_{g,n}} \rightarrow 0.
\end{equation}
The deformation space $T_{\mathcal{D}_{g,n}/\M_{g,n}}$ parameterizes deformations of a line bundle fixing the underlying curve, so its fiber over $\xi$ is $H^1(C, \O_C)$.  The map $\tau$ can be viewed fiberwise as 
\[\tau = (\tau_1, \tau_2): H^1(C, f^*\O_{\P^{r-1}}) \rightarrow \bigoplus_{i=1}^N H^1(L^{\otimes c_i}) \oplus H^1(f^*\O(1))^{\oplus r}.\]
Here, $\tau_1$ is the dual of the map $\bigoplus_{i=1}^NH^0(L^{\otimes - c_i} \otimes \omega)\rightarrow H^0(\omega)$ given by 
\begin{equation}
\label{tau1}
(q_1, \ldots, q_N) \mapsto \sum_{i=1}^N q_i s_i,
\end{equation}
and $\tau_2$ is dual to the map $H^0(f^*\O(-1) \otimes \omega)^{\oplus r} \rightarrow H^0(\omega)$ given by
\[(u_1, \ldots, u_r) \mapsto \sum_{j=1}^r u_j t_j;\]
in other words, $\tau_2$ arises via the Euler sequence on $\P^{r-1}$.  Thus, with (\ref{deformations}) in mind, we can view $\mathcal{O}b_{\mathcal{P}/\M_{g,n}}$ as $\text{coker}(\tau)$.  A straightforward argument using the quasihomogeneous Euler identity shows that the composition $\sigma \circ \tau$ vanishes, and therefore $\sigma$ descends to a cosection $\mathcal{O}b_{\mathcal{P}/\M_{g,n}} \rightarrow \O_{\mathcal{P}}.$
\end{remark}

Recall (Equation 4.3 of \cite{Li}) that the absolute obstruction sheaf $\mathcal{O}b_{\mathcal{P}}$ is defined as follows.  Let $q: \mathcal{P} \rightarrow \mathcal{D}_{g,n}$ be the projection, and form the distinguished triangle
\[q^*\mathbb{L}_{\mathcal{D}_{g,n}} \rightarrow \mathbb{L}_{\mathcal{P}} \rightarrow \mathbb{L}_{\mathcal{P}/\mathcal{D}_{g,n}} \xrightarrow{\delta} q^*\mathbb{L}_{\mathcal{D}_{g,n}}[1].\]
Let $\phi_{\mathcal{P}/\mathcal{D}_{g,n}}$ be the perfect obstruction theory of $\mathcal{P}$ relative to $\mathcal{D}_{g,n}$, so there is a map
\[\phi_{\mathcal{P}/\mathcal{D}_{g,n}} \circ \delta^{\vee}: q^*\mathbb{T}_{\mathcal{D}_{g,n}} \rightarrow \mathbb{T}_{\mathcal{P}/\mathcal{D}_{g,n}}[1] \rightarrow \mathbb{E}_{\mathcal{P}/\mathcal{D}_{g,n}}[1].\]
Then $H^0(\phi_{\mathcal{P}/\mathcal{D}_{g,n}} \circ \delta^{\vee})$ is the composite
\begin{equation}
\label{absolute}
q^*T_{\mathcal{D}_{g,n}} \rightarrow H^1(\mathbb{T}_{\mathcal{P}/\mathcal{D}_{g,n}}) \rightarrow H^1(\mathbb{E}_{\mathcal{P}/\mathcal{D}_{g,n}}) = \mathcal{O}b_{\mathcal{P}/\mathcal{D}_{g,n}}.
\end{equation}
The cokernel of (\ref{absolute}) is the absolute obstruction sheaf $\mathcal{O}b_{\mathcal{P}}$ of $\mathcal{P}$.

In order to apply the main theorem of \cite{Li} to conclude the existence of a localized virtual cycle, one must verify that $\sigma$ lifts to a cosection $\overline{\sigma}: \mathcal{O}b_{\mathcal{P}} \rightarrow \O_{\mathcal{P}}$.

\begin{lemma}
\label{abscosection}
The following composition is trivial:
\[H^1(\mathbb{T}_{\mathcal{P}/\mathcal{D}_{g,n}}) \rightarrow \mathcal{O}b_{\mathcal{P}/\mathcal{D}_{g,n}} \xrightarrow{\sigma} \O_{\mathcal{P}}.\]
Therefore, $\sigma$ lifts to $\overline{\sigma}: \mathcal{O}b_{\mathcal{P}} \rightarrow \O_{\mathcal{P}}$.
\begin{proof}
 The proof of this fact follows closely that of Lemma 3.6 of \cite{CL}.  First, we will need a slightly different description of $\sigma$.  First, note that the polynomial $\overline{W}$ defines a bundle homomorphism
\[h_1: \mathcal{Z} = \text{Vb}\left(\bigoplus_{i=1}^N\mathscr{L}^{\otimes c_i}_{\D_{g,n}} \oplus \mathscr{P}^{\oplus r}_{\D_{g,n}}\right) \rightarrow \text{Vb}(\omega_{\mathcal{C}_{\D_{g,n}}/\D_{g,n}}).\]
On tangent complexes, $h_1$ induces
\[dh_1: \Omega^{\vee}_{\mathcal{Z}/\mathcal{C}_{\mathcal{D}_{g,n}}} \rightarrow h_1^*\Omega^{\vee}_{\text{Vb}(\omega_{\mathcal{C}_{\D_{g,n}}/\D_{g,n}})}.\]
Pulling back $dh_1$ via the evaluation map $\mathfrak{e}$ defined above, one obtains
\[
\mathfrak{e}^*(dh_1): \mathfrak{e}^* \Omega^{\vee}_{\mathcal{Z}/\mathcal{C}_{\D_{g,n}}} \rightarrow \mathfrak{e}^* h_1^*\Omega^{\vee}_{\text{Vb}(\omega_{\mathcal{C}_{\D_{g,n}}/\D_{g,n}})},
\]
so applying $R^{\bullet}\pi_{\PP*}$ and taking first cohomology gives a map
\[\mathcal{O}b_{\PP/\D_{g,n}} \rightarrow \O_{\PP},\]
where we use the canonical isomorphisms
\[\mathfrak{e}^* \Omega^{\vee}_{\mathcal{Z}/\mathcal{C}_{\D_{g,n}}} \cong \bigoplus_{i=1}^N \mathscr{L}^{\otimes c_i} \oplus \mathscr{P}^{\oplus r},\]
\[\mathfrak{e}^*h_1^*\Omega^{\vee}_{\text{Vb}(\omega_{\mathcal{C}_{\D_{g,n}}/\D_{g,n}})} \cong \omega_{\mathcal{C}_{\PP}/\PP}.\]
One can check explicitly in coordinates that this coincides with the homomorphism $\sigma$ defined above.

Equipped with this description of $\sigma$, we are ready to prove the Lemma.  Let $\mathfrak{C}_{\omega} = C(\pi_*\omega_{\mathcal{C}_{\D_{g,n}}/\D_{g,n}})$ be the direct image cone (see Definition 2.1 of \cite{CL}), which parameterizes sections of $\omega$ on curves in $\D_{g,n}$.  This has a universal curve $\mathcal{C}_{\mathfrak{C}_{\omega}} = \mathcal{C}_{\D_{g,n}} \times_{\D_{g,n}} \mathfrak{C}_{\omega}$.  Let
\[\epsilon = \overline{W}(\mathfrak{s}_1, \ldots, \mathfrak{s}_N, \mathfrak{p}_1, \ldots, \mathfrak{p}_r) \in \Gamma(\mathcal{C}_{\PP}, \omega_{\mathcal{C}_{\PP}/\PP}),\]
which tautologically induces morphisms
\[\Phi_{\epsilon}: \PP \rightarrow \mathfrak{C}_{\omega}\]
and
\[\tilde{\Phi}_{\epsilon}: \mathcal{C}_{\PP} \rightarrow \mathcal{C}_{\mathfrak{C}_{\omega}}.\]
There are evaluation maps fitting into a commutative diagram of stacks of $\mathcal{C}_{\D_{g,n}}$ as follows:
\[\xymatrix{
\mathcal{C}_{\PP} \ar[r]^-{\mathfrak{e}}\ar[d]_{\tilde{\Phi_{\epsilon}}} & \mathcal{Z}\ar[d]^{h_1}\\
\mathcal{C}_{\mathfrak{C}_{\omega}}\ar[r]^-{\mathfrak{e}'} & \text{Vb}(\omega_{\mathcal{C}_{\D_{g,n}}/\D_{g,n}}).
}\]
Therefore, the following diagram of cotangent complexes is also commutative:
\begin{equation}
\label{tangentcx}
\xymatrix{
\pi_{\PP}^*\mathbb{T}_{\PP/\D_{g,n}}\ar@{=}[r]\ar[d] & \mathbb{T}_{\mathcal{C}_{\PP}/\mathcal{C}_{\D_{g,n}}}\ar[r]\ar[d] & \mathfrak{e}^*\Omega^{\vee}_{\mathcal{Z}/\mathcal{C}_{\D_{g,n}}}\ar[d]^{dh_1}\\
\pi_{\PP}^*\Phi_{\epsilon}^*\mathbb{T}_{\mathfrak{C}_{\omega}/\D_{g,n}}\ar@{=}[r] & \tilde{\Phi}^*_{\epsilon} \mathbb{T}_{\mathcal{C}_{\mathfrak{C}_{\omega}}/\mathcal{C}_{\D_{g,n}}} \ar[r] & \tilde{\Phi}^*_{\epsilon} \mathfrak{e}'^* \Omega^{\vee}_{\text{Vb}(\omega_{\mathcal{C}_{\D_{g,n}}/\D_{g,n}})/\mathcal{C}_{\D_{g,n}}}.
}
\end{equation}

Applying $R^1\pi_{\PP *}$ to the lower horizontal arrow yields the homomorphism
\[H^1(\Phi_{\epsilon}^*\mathbb{T}_{\mathfrak{C}_{\omega}/\D_{g,n}}) \rightarrow \Phi_{\epsilon}^*R^1\pi_{\mathfrak{C}_{\omega}*} \omega_{\mathcal{C}_{\mathfrak{C}_{\omega}}/\mathfrak{C}_{\omega}},\]
which is the pulback via $\Phi_{\epsilon}$ of the obstruction homomorphism in the perfect obstruction theory for $\mathfrak{C}_{\omega}$ over $\D_{g,n}$.  As observed in Equation 3.13 of \cite{CL}, this is trivial since $\mathcal{C}_{\mathfrak{C}_{\omega}} \rightarrow \mathcal{C}_{\D_{g,n}}$ is smooth.

Based on the new definition of $\sigma$ given above, it is clear that the composite whose vanishing we wish to show is obtained by applying $R^1\pi_{\PP *}$ to the composition from the upper left to the lower right of (\ref{tangentcx}).  Since we have now shown that the lower horizontal arrow becomes trivial, the proof is complete.
\end{proof}
\end{lemma}

Combining Lemmas \ref{degeneracy} and \ref{abscosection} with Theorem 1.1 of \cite{Li}, one finds that $\PP$ admits a localized virtual cycle $[\PP]^{\text{vir}}_{\text{loc}}$ supported on the degeneracy locus $\Mt_{g,(m_1, \ldots, m_n)}^d(\P^{r-1}, \beta) \subset \PP$ of $\sigma$.

\begin{definition}
The {\it virtual cycle} of the stack $\Mt_{g,(m_1, \ldots, m_n)}^d(\P^{r-1}, \beta)$ is defined as
\[ [\Mt_{g,(m_1, \ldots, m_n)}^d(\P^{r-1}, \beta)]^{\text{vir}} := [\PP]^{\text{vir}}_{\text{loc}}.\]
\end{definition}

\begin{remark}
\label{dichotomy}
To understand the motivation for this construction, and indeed for the LG/CY correspondence more generally, it is helpful to examine more closely the observation made above that $\PP$ embeds into the moduli space $\mathfrak{S}$ of sections associated to the diagram
\[\xymatrix{
\text{Vb}(\mathscr{L}_{\D_{g,n}}^{\otimes c_1} \oplus \cdots \mathscr{L}_{\D_{g,n}}^{\otimes c_N} \oplus \mathscr{P}_{\D_{g,n}}^{\oplus r}) \ar[r] & \mathcal{C}_{\mathcal{D}_{g,n}}\ar[d]\\
 & \mathcal{D}_{g,n}.
 }
\]
Specifically, $\mathcal{P}$ can be viewed as the substack of $\mathfrak{S}$ in which the $r$ sections of $\mathscr{P}$ parameterized by $\mathfrak{S}$ together define a stable map to $\P^{r-1}$.

If, on the other hand, we had considered the substack of $\mathfrak{S}$ in which the sections of $\mathscr{L}^{\otimes c_1}, \ldots, \mathscr{L}^{\otimes c_N}$ together define a stable map to $\P(c_1, \ldots, c_N)$, then the resulting moduli space would parameterize stable maps to this weighted projective space together with sections 
\[t_j \in H^0\left(f^*\O(-d) \otimes \omega \otimes \O\left(\sum(1-m_i)[x_i]\right)\right)\]
for $j = 1, \ldots, r$, assuming that the Gorenstein condition (\ref{Gorenstein}) is satisfied.  The cosection $\sigma$ is still defined on this new moduli space, and its degeneracy locus is the moduli space of stable maps to the complete intersection $X_{\overline{W}} \subset \P(c_1, \ldots, c_N)$, as Chang-Li prove in \cite{CL} for the case of the quintic.\footnote{In fact, much more is proved in \cite{CL}, since it is not at all obvious {\it a priori} that the localized virtual cycle obtained by the above method agrees with the usual virtual cycle on the moduli space of stable maps.}
\end{remark}

\begin{remark}
Because we have used the cosection construction as opposed to the Witten top Chern class construction of \cite{ChiodoW} and \cite{PV}, it is not at all obvious that our correlators agree in the case of the quintic with those defined in \cite{CR}.  However, the equivalence of all existing constructions of the FJRW virtual cycle is proved in \cite{CLL}.
\end{remark}

\subsubsection{Virtual dimension}

Let $\xi = (C, f, L, \varphi, s_1, \ldots, s_N) \in \mathcal{P}$.  The virtual dimension of $\mathcal{P}/\mathcal{D}_{g,n}$ at $\xi$ is
\[h^0(L^{\otimes c_1} \oplus \cdots \oplus L^{\otimes c_N} \oplus f^*\O(1)^{\oplus r}) - h^1(L^{\otimes c_1} \oplus \cdots \oplus L^{\otimes c_N} \oplus f^*\O(1)^{\oplus r}),\]
and an easy Riemann-Roch computation using (\ref{reification}) shows that this equals
\[(N-r)(1-g) + rn - \sum_{i=1}^n \sum_{j=1}^N \frac{c_jm_{i,j}}{d}.\]
Since
\begin{align*}
\text{vdim}(\mathcal{D}_{g,n}) &= \text{vdim}(\mathcal{D}_{g,n}/\M_{g,n}) + \text{vdim}(\M_{g,n})\\
&= (h^0(\O_C) - 1) + 3g-3+n\\
&=4g - 4 +n,
\end{align*}
we find that the virtual dimension of $\mathcal{P}/\M_{g,n}$ at $\xi$ equals
\[(N-r-4)(1-g) + (r+1)n - \sum_{i=1}^n\sum_{j=1}^N \frac{c_jm_{i,j}}{d} = \text{vdim}(\M_{g,n}(\P^{r-1}, \beta)) + \sum_{j=1}^N \chi(L^{\otimes c_j}).\]

\subsubsection{Virtual cycle in genus zero}

In genus zero, the definition of the virtual cycle simplifies substantially, under the Gorenstein condition
\begin{equation}
\label{Gorenstein}
c_j|d \; \; \text{for all } j.
\end{equation}

Indeed, if this hypothesis is satisfied and $L$ is a line bundle satisfying the requirements of $\mathcal{P}$, then the bundles  $L^{\otimes c_j}$ have no global sections.  To see this, one simply must compute the degree of such a line bundle using the fact that on each irreducible component $Z$ of the source curve $C$,
\[L^{\otimes c_j}|_Z^{\otimes d/c_j} \cong \omega_{\log} \otimes f^*\O(-1) \otimes \O\left(-\sum_{i=1}^n m_{i,j} [x_j]\right).\]
Here, the $x_j$ are the special points on $Z$ and the $m_{i,j}$ are the multiplicities of $L^{\otimes c_j}$ at those special points, and we are once again using that the multiplicities at all marked points are nonzero.  This equation implies that the degree of $L^{\otimes c_j}|_Z$ is negative, so if $C$ is itself irreducible, it follows that $L^{\otimes c_j}$ has no global sections.  If $C$ is reducible, the claim still follows by an easy inductive argument using the fact that $\text{deg}(L^{\otimes c_j}|_Z) < k-1$, where $k$ is the number of points at which $Z$ meets the rest of $C$.

Because of this observation, $\mathcal{P} = \Mt_{0,(m_1, \ldots, m_n)}^d (\P^{r-1}, \beta)$.  In this situation, the cosection localized virtual cycle is the same as the ordinary virtual cycle of $\Mt_{0,(m_1, \ldots, m_n)}^d(\P^{r-1}, \beta)$.  Furthermore, abbreviating $\Mt = \Mt_{0,(m_1, \ldots, m_n)}^d(\P^{r-1}, \beta)$ and $Y = \M_{0,n}(\P^{r-1}, \beta)$, the smoothness of the moduli space in this case implies that
\[[\Mt]^{\vir} = c_{top}(\mathcal{O}b_{\Mt/Y}) \cap [\M_{0,n}(\P^{r-1}, \beta)],\]
where $[\M_{0,n}(\P^{r-1}, \beta)]$ denotes the pullback of the fundamental class on $Y$ to $\Mt$ under the map that forgets $L$.  Using the exact sequence
\[T_{\Mt/\mathcal{D}_{g,n}} \xrightarrow{\sim} T_{Y/\mathcal{D}_{g,n}} \rightarrow \O b_{\Mt/Y} \rightarrow \O b_{\Mt/\mathcal{D}_{g,n}} \rightarrow 0,\]
one finds that $\O b_{\Mt/Y} = R^1\pi_*(\mathcal{T}^{\otimes c_1} \oplus \cdots \oplus \mathcal{T}^{\otimes c_N})$, in which $\mathcal{T}$ is the universal line bundle on $\Mt$.  Thus, we obtain the formula
\[[\Mt]^{\vir} = c_{top}(R^1\pi_*(\mathcal{T}^{\otimes c_1} \oplus \cdots \oplus \mathcal{T}^{\otimes c_N}))\cap [\M_{0,n}(\P^{r-1}, \beta)]\]
for the virtual cycle in genus zero.

\subsection{Correlators}

In analogy to Gromov-Witten theory, correlators will be defined as integrals over the moduli space against the virtual cycle.

\subsubsection{Evaluation maps and psi classes}

The classes that we integrate will come from two places.  First, there are evaluation maps
\[\ev_i: \Mt_{g,n}^d(\P^{r-1}, \beta) \rightarrow \P^{r-1} \hspace{1.5cm} i=1, \ldots, n,\]
given by $(C, f, L, \varphi) \mapsto f(x_i)$, where $x_i \in C$ is the $i$th marked point.  Therefore, we can pull back cohomology classes on $\P^{r-1}$ to obtain classes on the hybrid moduli space.

Second, there are classes
\[\psi_i \in H^2(\Mt_{g,n}^d(\P^{r-1},\beta))\]
for $i=1, \ldots, n$, defined in the same way as in Gromov-Witten theory.  Namely, $\psi_i$ is the first Chern class of the (orbifold) line bundle whose fiber at a point of the moduli space is the cotangent line to the orbifold curve at the $i$th marked point.  Note that this differs from the definition of $\psi_i$ used in \cite{CR}, in which the cotangent line was always taken to the underlying curve; we will denote these ``coarse" psi classes by $\overline{\psi}_i$.  The two are related by
\[\overline{\psi}_i = d\psi_i.\]

\subsubsection{Definition of correlators in the narrow case}

We will only define correlators when all insertions are drawn from the narrow sectors of the state space.

\begin{definition}
\label{correlatorsdef}
Choose $\phi_1, \ldots, \phi_n\in \mathcal{H}_{hyb}(W_1, \ldots, W_r)$ from the narrow sectors and $a_1, \ldots, a_n \geq 0$.  As explained in Section \ref{broadnarrow}, each $\phi_i$ can be viewed as an element of $H^*(\P^{r-1})$.  Each also defines an element $\gamma_i \in J$ indicating the twisted sector from which it is drawn, and we let $m_i \in \{0, 1, \ldots, d-1\}$ be such that
\[\gamma_i = (e^{2\pi i \frac{m_ic_1}{\overline{d}}}, \ldots, e^{2\pi i \frac{m_ic_N}{\overline{d}}}, 1, \ldots, 1).\]
Define the associated {\it hybrid model correlator} $\langle \tau_{a_1}(\phi_1) \cdots \tau_{a_n}(\phi_n) \rangle^{hyb}_{g,n,\beta}$ to be
\begin{equation}
\label{correlators}
\frac{d(-1)^D}{\deg(\rho)}\displaystyle\int_{[\Mt_{g,(m_1, \ldots, m_n)}^d(\P^{r-1},\beta)]^{vir}} \ev_1^*(\phi_1)\psi_1^{a_1} \cdots \ev_n^*(\phi_n)\psi_n^{a_n},
\end{equation}
where
\[D = \sum_{i=1}^N\bigg(h^1(L^{\otimes c_i}) - h^0(L^{\otimes c_i})\bigg)\]
and $\deg(\rho)$ denotes the degree of the map
\[\rho: \Mt_{g,(m_1, \ldots, m_n)}^d(\P^{r-1},\beta) \rightarrow \M_{g,n}(\P^{r-1}, \beta)\]
given by forgetting the $\overline{W}$-structure and passing to the coarse underlying source curve.
\end{definition}

The strange-looking sign choice in this definition is a matter of convenience, following equation (50) of \cite{FJR}.  In genus zero under the Gorenstein condition (\ref{Gorenstein}), $D$ is precisely the rank of the obstruction bundle and $\deg(\rho) = \frac{1}{d}$ whenever the substratum over which we are integrating is nonempty (see Equation (26) of \cite{FJR}).  Thus, the above is equivalent in such cases to
\[d^2 \int_{[\Mt^d_{0,\mathbf{m}}(\P^{r-1}, \beta)]^{\vir}} \ev_1^*(\phi_1)\psi_1^{a_1} \cdots \ev_n^*(\phi_n)\psi_n^{a_n} c_{top}((R^1\pi_*(\mathcal{T}^{\otimes c_1} \cdots \mathcal{T}^{\otimes c_N}))^{\vee}),\]
where $\mathbf{m} = (m_1, \ldots, m_n)$.

\subsubsection{Broad insertions}

The easiest way to extend the above theory to allow for broad insertions is to set a correlator to zero if any of its insertions comes from a broad sector.  In order to ensure that the resulting theory satisfies the decomposition property required of cohomological field theories, though, it is necessary to verify a Ramond vanishing property.  This holds whenever the Gorenstein condition (\ref{Gorenstein}) is satisfied.\footnote{The argument below was substantially simplified by a suggestion of A. Chiodo.}

\begin{proposition}[Ramond vanishing]
Suppose that for all $i$ and $j$, $c_j|d$ and $m_{i,j} \neq 0$.  Let $D \subset \Mt^d_{0, \mathbf{m}}(\P^{r-1}, \beta)$ be a boundary stratum whose general point is a source curve with a single, broad node.  Then
\begin{equation}
\label{ramondvanishing}
\int_D \ev_1^*(\phi_1)\psi_1^{a_1} \cdots \ev_n^*(\phi_n)\psi_n^{a_n} c_{top}((R^1\pi_*(\mathcal{T}^{\otimes c_1} \oplus \cdots \oplus \mathcal{T}^{\otimes c_N}))^{\vee}) = 0
\end{equation}
for any $a_1, \ldots, a_n \in \Z^{\geq 0}$ and any $\phi_1, \ldots, \phi_n \in H^*(\P^{r-1})$.
\begin{proof}
Let $C = C_1 \sqcup C_2$ be the decomposition of a fiber of $\pi$ in $D$ into irreducible components, and let $n$ be the node at which the components meet.  Consider the normalization exact sequence
\[0 \rightarrow \O_C \rightarrow \O_{C_1} \oplus \O_{C_2} \rightarrow \O_n \rightarrow 0.\]
Tensor with $L^{\otimes c_1} \oplus \cdots \oplus L^{\otimes c_N}$ and take the associated long exact sequence in cohomology to obtain the following equality in $K$-theory:
\[R^1\pi_*\left(\oplus \mathcal{T}^{\otimes c_j}\right) = R^0\pi_*(\oplus \mathcal{T}^{\otimes c_j}|_n) + R^1\pi_*(\oplus \mathcal{T}^{\otimes c_j}|_{C_1}) + R^1\pi_*(\oplus \mathcal{T}^{\otimes c_j}|_{C_2}).\]
Here, we use that $H^0(L^{\otimes c_j}|_{C_i}) = 0$ for all $i$ and $j$, as an easy degree computation shows.

The key point is that, since $n$ is broad, orbifold sections of $L^{\otimes c_j}$ over $n$ are the same as ordinary sections over the coarse underlying curve.  More precisely, let $N: D \rightarrow \mathscr{C}$ be the section of the universal curve defined by the node $n$.  Then
\[c_{top}(R^0\pi_*(\oplus \mathcal{T}^{\otimes c_j}|_n)) = N^*c_{top}(\oplus \mathcal{T}^{\otimes c_j}),\]
and since $(\mathcal{T}^{\otimes c_j})^{\otimes d} \cong \omega_{\log}^{\otimes c_j} \otimes f^*\O(-c_j)$ for each $j$, the above equals
\[\prod_{j=1}^N \frac{1}{d}\left(c_{top}(\omega_{\log}^{\otimes c_j}|_n) + c_{top}(N^*f^*\O(-c_j))\right).\]
In this expression, $c_{top}(\omega_{\log}^{\otimes c_j}|_n) = 0$, since the restriction of $\omega_{\log}$ to the locus of nodes is trivial.  Furthermore, $f \circ N = \ev_n$, so we can rewrite the above as
\[\prod_{j=1}^N \frac{1}{d} ev_n^*c_{top}(\O(-c_j)) = \frac{1}{d^N} \ev_n^*c_{top}(\O(-c_j)^{\oplus N}),\]
which is zero because $\O(-c_j)^{\oplus N}$ is an $N$-dimensional bundle on an $r$-dimensional space and $N > r$.

It follows that one of the summands in the expression for $R^1\pi_*(\oplus \mathcal{T}^{\otimes c_j})$ has trivial top Chern class, so the integral in (\ref{ramondvanishing}) vanishes.
\end{proof}
\end{proposition}

\begin{remark}
This definition of the broad correlators seems initially ad hoc.  However, analogously to Proposition 2.4.5 of \cite{CR}, it is possible to unify the broad and narrow cases in genus $0$ into a single geometric definition by slightly modifying the moduli space.
\end{remark}

\subsubsection{Multiplicity conditions}

Certain tuples of multiplicities correspond to empty components of the moduli space, so the resulting correlators clearly vanish. Indeed, (\ref{reification}) and the subsequent discussion imply that if $m_1, \ldots, m_n$ are as in Definition \ref{correlatorsdef}, then $\langle \tau_1(\phi_1) \cdots \tau_n(\phi_n)\rangle^{hyb}_{g,n,\beta}$ vanishes unless
\begin{equation}
\label{selection}
2g - 2 + n - \beta - \sum_{i=1}^n m_i \equiv 0 \mod d.
\end{equation}
This selection rule will be useful later.

\section{Proof of the correspondence in genus zero}
\label{correspondence}

In both Gromov-Witten theory and the hybrid model, the genus-zero theory can be realized as a Lagrangian cone in a certain symplectic vector space.  Because the genus-zero hybrid invariants are described via a top Chern class, they fit into the framework of twisted invariants described in \cite{Coates}, and Givental's quantization formalism provides a tool for realizing them in terms of the corresponding {\it untwisted} theory, which is essentially the Gromov-Witten theory of projective space.  The following section describes this process in detail and uses it to prove the LG/CY correspondence in the two cases of interest.

\subsection{Givental's formalism}

For the sake of expository clarity, we will describe the setup in the case of the cubic singularities first, commenting briefly on the requisite modifications for the quadric case at the end.

\subsubsection{The symplectic vector spaces}

It is convenient to modify the state space slightly, replacing the broad sector with another copy of $H^*(\P^1)$ to obtain
\[H_{hyb} = H^*_{0}(\P^1) \oplus H_{1}^*(\P^1) \oplus H_{2}^*(\P^1).\]
The subscripts denote the multiplicities to which the summands correspond.  This modification does not affect the correlators, since they vanish when any insertion is broad.  We will write $\phi^{(h)}$ for an element $\phi \in H^*(\P^1)$ coming from the summand $H^*_h(\P^1)$.

This vector space is equipped with a nondegenerate inner product (or {\it Poincar\'e pairing}), denoted $(\; , \;)_{hyb}$ and defined as
\[(\Theta_1, \Theta_2)_{hyb} = \langle \tau_0(\Theta_1)\; \tau_0(\Theta_2)\; 1^{(1)} \rangle^{hyb}_{0,3,0}.\]  The symplectic vector space we will consider is
\[\mathcal{V}_{hyb} = H_{hyb} \otimes \C((z^{-1})),\]
with the symplectic form $\Omega_{hyb}$ given by
\[\Omega_{hyb}(f,g) = \text{Res}_{z=0}\bigg((f(-z), g(z))_{hyb} \bigg).\]
This induces a polarization $\mathcal{V}_{hyb} = \mathcal{V}_{hyb}^+ \oplus \mathcal{V}_{hyb}^-$, where $\mathcal{V}_{hyb}^+ = H_{hyb} \otimes \C[z]$ and $\mathcal{V}_{hyb}^- = z^{-1}H_{hyb} \otimes \C[[z^{-1}]]$.  Thus, we can identify $\mathcal{V}_{hyb}$ as a symplectic manifold with the cotangent bundle to $\mathcal{V}_{hyb}^+$.  An element of $\mathcal{V}_{hyb}$ can be expressed in Darboux coordinates as $\sum_{k \geq 0} q_k^{\alpha} \phi_{\alpha}z^k + \sum_{\ell \geq 0} p_{\ell, \beta} \phi^{\beta} (-z)^{-\ell -1}$, where $\{\phi_{\alpha}\}$ is a basis for $H_{hyb}$.

Analogously, there is a symplectic vector space on the Gromov-Witten side \cite{CR} \cite{Coates}.  The restriction to narrow states is mirrored in that setting by the restriction to cohomology classes pulled back from the ambient projective space, which are the only ones that give nonzero correlators.  Let $H_{GW}$ denote the vector space of such classes:
\[H_{GW} = H^{even}(X_{3,3}) = \bigoplus_{h=0}^3 [H^h] \C,\]
where $H$ is the restriction to $X_{3,3}$ of the hyperplane class on the ambient projective space.\footnote{Of course, to be completely symmetric, we might want to add an additional two-dimensional summand to $H_{GW}$, as we did for $H_{hyb}$, and define the correlators to vanish if any insertion comes from this summand.  Since we will not be doing any computations on the Gromov-Witten side, we will ignore this asymmetry and leave $H_{GW}$ as above.}  The symplectic vector space $\mathcal{V}_{GW}$ on the Gromov-Witten side is defined as above, and the usual Poincar\'e pairing on $H_{GW}$ induces a symplectic form in the same way.

\subsubsection{The potentials}

Defining the correlators in the hybrid theory as above, the generating function for the genus-$g$ invariants is
\[\mathcal{F}^g_{hyb}(\t, z) = \sum_{n,d} \frac{Q^d}{n!} \langle \t(\psibar), \ldots, \t(\psibar) \rangle^{hyb}_{g,n,d},\]
where $\t = t_0 + t_1z + t_2z^2 + \cdots \in H_{hyb}[[z]]$.  These generating functions fit together into a total-genus descendent potential
\[\mathcal{D}_{hyb} = \exp\left( \sum_{g \geq 0} \hbar^{g-1} \mathcal{F}^g_{hyb}\right).\]

In the same way, one can define a generating function for the genus-$g$ Gromov-Witten invariants of the corresponding complete intersection,
\[\mathcal{F}^g_{GW}(\t,z) = \sum_{n,d} \frac{Q^d}{n!} \langle \t(\psi), \ldots, \t(\psi) \rangle_{g,n,d}^{GW},\]
where $\t = t_0 + t_1z + t_2z^2 + \cdots \in H_{GW}[[z]]$.  These, too, fit together into a total-genus descendent potential $\mathcal{D}_{GW}$.

\subsubsection{The Lagrangian cones}

In the Gromov-Witten setting, the dilaton shift
\[q_k^{\alpha} = t_k^{\alpha} - 1\cdot z\]
makes $\mathcal{F}^0_{GW}$ into a power series in the Darboux coordinates $q_k^{\alpha}$, where $1$ denotes the constant function $1$ in $H^0$.  In this way, the genus-zero Gromov-Witten theory is encoded by a Lagrangian cone
\[\mathcal{L}_{GW} = \{(\mathbf{q}, \mathbf{p}) \; | \; \mathbf{p} = d_{\mathbf{q}} \mathcal{F}^0_{GW}\} \subset \mathcal{V}_{GW},\]
where we use the Darboux coordinates $(\mathbf{q}, \mathbf{p})$ defined above to identify $\mathcal{V}_{GW}$ with the cotangent bundle to its Lagrangian subspace $\mathcal{V}_{GW}^+$.  As proved in \cite{Coates}, $\mathcal{L}_{GW}$ is a Lagrangian cone whose tangent spaces satisfy the geometric condition
\begin{equation}
\label{tangent}
zT_f \mathcal{L}_{GW} = \mathcal{L}_{GW} \cap T_f \mathcal{L}_{GW}
\end{equation}
at any point.

The same story holds in the hybrid model, but it is important to note that in the dilaton shift
\[q_k^{\alpha} = t_k^{\alpha} - 1^{(1)} \cdot z,\]
the {\it unit is the constant function $1$ from the summand of the state space corresponding to multiplicity-$1$ insertions}.  Under this dilaton shift, we again have that $\mathcal{F}^0_{hyb}$ is a function of $\mathbf{q} \in \mathcal{V}_{hyb}^+$ and hence we can define
\[\mathcal{L}_{hyb} = \{(\mathbf{q}, \mathbf{p}) \; | \; \mathbf{p} = d_{\mathbf{q}} \mathcal{F}^0_{hyb} \}\subset \mathcal{V}_{hyb}.\]
Since the hybrid theory also satisfies the string equation, dilaton equation, and topological recursion relations, the same geometric condition holds for this cone as for the Lagrangian cone of Gromov-Witten theory.

On either the Gromov-Witten or the hybrid side, we define the $J$-function
\[J_{hyb/GW}(z,\t) = 1z +\t + \sum_{n,d} \frac{1}{n!}\left\langle \t, \ldots, \t, \frac{\phi_{\alpha}}{z-\psi}\right\rangle_{0,n+1,d}^{hyb/GW}\phi^{\alpha},\]
where $\phi_{\alpha}$ ranges over a basis for $H_{hyb/GW}$ with dual basis $\phi^{\alpha}$.  In other words, $J(-z,t)$ is the intersection of the Lagrangian cone with the slice $\{-1z + \t + \mathcal{V}^-\} \subset \mathcal{V}_{hyb/GW}$.  It is a well-known consequence of (\ref{tangent}) that this slice determines the rest of the Lagrangian cone, so the $J$-function specifies the entire genus-zero theory.

\subsubsection{Twisted theory}

The strategy for determining $J_{hyb}$ is to introduce parameters that will interpolate between the hybrid invariants and the ordinary Gromov-Witten invariants of projective space.  One can always define a multiplicative characteristic class $K_0(X) \rightarrow H^*(X; \C)$ by
\[x \mapsto \exp\left(\sum_{k \geq 0} s_k \ch_k(x)\right).\]
When $s_k = 0$ for all $k \geq 0$, the result is a constant map sending every $K$-class to the fundamental class, while if we set
\begin{equation}
\label{sk}
s_k = \begin{cases} -6\ln(\lambda) & k=0\\ & \\ \displaystyle\frac{6(k-1)!}{\lambda^k} & k > 0,\end{cases}
\end{equation}
then the resulting class satisfies
\[\exp\left(\sum_{k \geq 0} s_k \ch_k(-[V])\right) = e_{\C^*}(V^{\vee})^6\]
for any vector bundle $V$ equipped with the natural $\C^*$ action scaling the fibers.  (The reason for passing to equivariant cohomology is to ensure that the above is invertible.)  We will typically denote
\[c(x) = \exp\left(\sum_{k \geq 0} s_k \ch_k(x)\right)\]
when the parameters $s_k$ are taking unspecified values.

Extend the hybrid model state space to
\[H_{tw} = \big (H^*_0(\P^1) \oplus H^*_1(\P^1)\oplus H^*_2(\P^1)\big) \otimes R,\]
where
\[R = \C[\lambda] [[s_0, s_1, \ldots ]].\]
Then, for any $\phi_1, \ldots, \phi_n \in H_{tw}$ and $a_1, \ldots, a_n \in \Z^{\geq 0}$, define the corresponding {\it twisted hybrid invariant} $\langle \tau_{a_1}(\phi_1), \ldots, \tau_{a_n}(\phi_n)\rangle^{tw}_{g,n,d} $ by
\[\frac{3}{\text{deg}(\rho)} \int_{\rho^*[\M_{g,n}(\P^1,d)]^{\vir}} \ev_1^*(\phi_1)\psi_1^{a_1} \cdots \ev_n^*(\phi_n)\psi_n^{a_n} \; c(R\pi_*\mathcal{T}),\]
where $\mathcal{T}$ denotes the universal line bundle on the universal curve over $\Mt^3_{g,n}(\P^1, d)$, $\rho: \Mt^3_{g, \mathbf{m}}(\P^1,d)\rightarrow \M_{g,n}(\P^1,d)$ is as in Section \ref{correlatorsdef}.  We will sometimes adopt the notation of \cite{Coates} and write the above as
\[\langle \tau_{a_1}(\phi_1), \ldots, \tau_{a_n}(\phi_n); c(R\pi_*\mathcal{T})\rangle_{g,n,d},\]
or more generally, write a cohomology class on the universal curve after a semicolon to indicate that it is part of the integrand but is neither a $\psi$ class nor pulled back from the target space.

Via these invariants, $H_{tw}$ is equipped with a pairing extending the pairing on $H_{hyb}$:
\[(\Theta_1, \Theta_2)_{tw} = \langle \Theta_1, \Theta_2, 1^{(1)}\rangle^{tw}_{0,3,0}.\]
We can then set $\mathcal{V}_{tw} = H_{tw} \otimes \C((z^{-1}))$, and this is a symplectic vector space under the symplectic form induced by the twisted pairing.  The definitions of the genus-$g$ potential, total descendent potential, and Lagrangian cone all generalize directly, and we thus obtain the twisted Lagrangian cone $\mathcal{L}_{tw} \subset \mathcal{V}_{tw}$.  It is no longer obvious that this is indeed a Lagrangian cone, but this will follow from Proposition \ref{symplectic}.

\subsubsection{Untwisted theory}

Let $\mathcal{V}_{un}$ denote the symplectic vector space obtained by setting $s_k = 0$ for all $k \geq 0$, and similarly $H_{un}$ and $\mathcal{L}_{un}$.  Note that $\mathcal{L}_{un}$ encodes the correlators $\langle \tau_{a_1}(\phi_1), \ldots, \tau_{a_n}(\phi_n)\rangle^{un}_{0,n,d}$, which are given by
\[\frac{3}{\text{deg}(\rho)} \int_{\rho^*[\M_{0,n}(\P^1, d)]^{\vir}} \ev_1^*(\phi_1) \psi_1^{a_1} \cdots \ev_n^*(\phi_n) \psi_n^{a_n}.\]
When the selection rule (\ref{selection}) is satisfied so that the component of the hybrid moduli space over which we are integrating is nonempty, these are simply three times the Gromov-Witten invariants of $\P^1$.  In particular, the untwisted $J$-function is known explicitly.

We will use the untwisted Lagrangian cone to determine the cone $\mathcal{L}_{hyb}$.  This can be viewed as a two-step procedure.  First, $\mathcal{L}_{hyb}$ can be obtained from $\mathcal{L}_{tw}$ by taking a limit $\lambda \rightarrow 0$ and setting the parameters $s_k$ to the values in (\ref{sk}), so that
\[c(R\pi_*\mathcal{T}) = c(-R^1\pi_*\mathcal{T}) = c_{top}((R^1\pi_*\mathcal{T})^{\vee})^6,\]
which is what appears in the hybrid model correlators.  Then, Proposition \ref{symplectic} demonstrates that $\mathcal{L}_{tw}$ can in turn be recovered from $\mathcal{L}_{un}$.

\subsubsection{The quadric singularities}

All of the above is defined analogously in the other example of interest.  In that case,
\[H_{hyb} = H_0^*(\P^3) \oplus H^*_1(\P^3).\]
The hybrid Poincar\'e pairing is defined by the exact same formula, and we obtain a symplectic vector space $\mathcal{V}_{hyb} = H_{hyb} \otimes \C((z^{-1}))$.  The symplectic vector space on the Gromov-Witten side is now $\mathcal{V}_{GW} = H_{GW} \otimes \C((z^{-1}))$, where
\[H_{GW} = H^{even}(X_{2,2,2,2}) = \bigoplus_{h=0}^3 [H^h]\C\]
and $H$ is the restriction to $X_{2,2,2,2}$ of the hyperplane class on $\P^7$.  The genus-$g$ generating functions and total-genus descendent potentials on both the hybrid and the Gromov-Witten side are defined just as before, and again the genus-$0$ theory on each side is encoded by a Lagrangian cone which is determined by the slice cut out by a $J$-function.

A twisted theory is again introduced, though now the values of $s_k$ that give the hybrid theory are
\begin{equation}
\label{sk2}
s_k = \begin{cases} -8\ln(\lambda) & k=0\\ & \\ \displaystyle\frac{8(k-1)!}{\lambda^k} & k > 0,\end{cases}
\end{equation}
since the virtual class in genus $0$ is $c_{top}((R^1\pi_*\mathcal{T})^{\vee})^8$ in this case.  The state space is extended to
\[H_{tw} = (H^*_0(\P^3) \oplus H^*_1(\P^3)) \otimes R\]
for $R = \C[\lambda][[s_0, s_1, \ldots]]$, and twisted hybrid invariants are defined as
\[\frac{2}{\text{deg}(\rho)} \int_{\rho^*[\M_{g,n}(\P^3, d)]^{\vir}} \ev_1^*(\phi_1)\psi_1^{a_1} \cdots \ev_n^*(\phi_n)\psi_n^{a_n} \; c(R\pi_*\mathcal{T}),\]
for $\phi_1, \ldots, \phi_n \in H_{tw}$ and $a_1, \ldots, a_n \in \Z^{\geq 0}$.  These permit the definition of the twisted Poincar\'e pairing and hence the twisted symplectic vector space.  When $\lambda \rightarrow 0$ and the parameters $s_k$ are set to the values in (\ref{sk2}), we obtain the hybrid theory for the quadric singularity, while the untwisted theory (when $s_k =0$ for all $k$) gives two times the Gromov-Witten theory of $\P^3$.

\subsection{Lagrangian cone for the Landau-Ginzburg theory}

Recall that the Bernoulli polynomials $B_n(x)$ are defined by the generating function
\[\sum_{n=0}^{\infty} B_n(x) \frac{z^n}{n!} = \frac{ze^{zx}}{e^z-1}.\]

\begin{proposition}
\label{symplectic}
\begin{enumerate}[(a)]
\item Let $\mathcal{V}_{tw}$ denote the symplectic vector space associated to the cubic singularities $W_1(x_1, \ldots, x_6), \ldots, W_3(x_1, \ldots, x_6)$, and let $\Delta: \mathcal{V}_{un} \rightarrow \mathcal{V}_{tw}$ be the symplectic transformation
\[\Delta = \bigoplus_{\ell=0}^2 \exp\left( \sum_{\substack{ k \geq 0\\ m \geq 0}} s_k \frac{B_m(\frac{\ell}{3})}{m!} \exp(-\sfrac{H^{(\ell)}}{3})_{k+1-m} z^{m-1}\right).\]
Then $\mathcal{L}_{tw} = \Delta(\mathcal{L}_{un})$.
\item Let $\mathcal{V}_{tw}$ denote the symplectic vector space associated to the quadric singularities $V_1(x_1, \ldots, x_8),$ $\ldots, V_4(x_1, \ldots, x_8)$, and let $\Delta: \mathcal{V}_{un} \rightarrow \mathcal{V}_{tw}$ be the symplectic transformation
\[\Delta = \bigoplus_{\ell=0}^1 \exp\left( \sum_{\substack{ k \geq 0\\ m \geq 0}} s_k \frac{B_m(\frac{\ell}{2})}{m!} \exp(-\sfrac{H^{(\ell)}}{2})_{k+1-m} z^{m-1}\right).\]
Then $\mathcal{L}_{tw} = \Delta(\mathcal{L}_{un})$.
\end{enumerate}
\end{proposition}

\begin{proof}

We will prove part (a) of the Proposition; the proof of part (b) is almost identical, so we will omit it.  Our proof is modeled closely after that of Theorem 4.2.1 of \cite{Tseng}, which in turn uses the main idea of Proposition 1.6.3 of \cite{Coates}.

Let us begin by reducing the statement to something more concrete.  According to the theory of Givental quantization, the desired statement $\mathcal{L}_{tw} = \Delta(\mathcal{L}_{un})$ will be implied if we can demonstrate that $\mathcal{D}_{tw} = \widehat{\Delta}(\mathcal{D}_{un})$.  In fact, it suffces to show that $\mathcal{D}_{tw} \approx \Delta(\mathcal{D}_{un})$, where the symbol $\approx$ denotes equality up to a scalar factor in $R$, since $\mathcal{L}_{tw}$ is a cone and hence is unaffected by scalar multiplication.  Furthermore, $\mathcal{D}_{tw} \approx \Delta(\mathcal{D}_{un})$ if and only if this holds after differentiating both sides with respect to $s_k$ for all $k$.  If $C_k: \mathcal{V}_{un} \rightarrow \mathcal{V}_{tw}$ denotes the infinitesimal symplectic transformation\footnote{The fact that this transformation is infinitesimal symplectic is required for the quantization to be defined; it follows from the same argument as in Lemma 4.1.3 of \cite{Tseng}.}
\[C_k = \bigoplus_{\ell= 0}^2 \left(\sum_{m \geq 0} \frac{B_m(\frac{\ell}{3})}{m!} \exp(-\sfrac{H^{(\ell)}}{3})_{k+1-m} z^{m-1}\right),\]
then we have $\Delta = \exp(\sum_{k \geq 0} s_k C_k)$, so $\mathcal{D}_{tw} \approx \Delta(\mathcal{D}_{un})$ is equivalent to the system of differential equations
\[\frac{\d \mathcal{D}_{tw}}{\d s_k} \approx \widehat{C_k} \mathcal{D}_{tw} + \mathcal{C} \mathcal{D}_{un}\]
for all $k$, where $\mathcal{C}$ is the cocycle coming from commuting the $\hat{z}$ terms of $\widehat{\Delta}$ past the $\widehat{1/z}$ term of $\widehat{C_k}$; see the discussion in Section 1.3.4 of \cite{Coates}.  Since we only seek equality up to a scalar factor, we can absorb the cocyle into the definition of $C_k$ and prove that $\d \mathcal{D}_{tw}/\d s_k \approx \widehat{C_k} \mathcal{D}_{tw}$.  We will use the orbifold Grothendieck-Riemann-Roch (oGRR) formula\footnote{An alternative, and perhaps shorter, proof can be obtained by passing to the coarse underlying curve and applying the usual GRR formula, as in \cite{CZ}.} (see Appendix A of \cite{Tseng} for the statement) to determine $\d \mathcal{D}_{tw}/\d s_k$ and identify it with an explicit expression for $\widehat{C_k}$.

Specifically, we have
\begin{equation}
\label{deriv}
\frac{\d \mathcal{D}_{tw}}{\d s_k} = \sum_{g,n,d} \frac{Q^d \hbar^{g-1}}{n!} \langle \t, \ldots, \t; \ch_k(R\pi_*\mathcal{T}) \; c(R\pi_*\mathcal{T})\rangle_{g,n,d} \mathcal{D}_{tw},
\end{equation}
and oGRR will be used to compute the contribution from $\ch_k(R\pi_*\mathcal{T})$.  As remarked in Section 7.3 of \cite{Tseng}, the moduli stack $\Mt^3_{g,n}(\P^1, d)$ can be embedded in a smooth stack $\M$ over which there exists a family $\mathcal{U}$ of orbicurves pulling back to the universal family $\mathscr{C}$ over $\Mt^3_{g,n}(\P^1, d)$.  Therefore, we lose no information if we assume that the moduli stack itself is smooth, in which case $\ch(R\pi_*\mathcal{T}) = \widetilde{\ch}(R\pi_*\mathcal{T})$ and oGRR states that
\begin{equation}
\label{oGRR}
\ch(R\pi_*\mathcal{T}) = I\pi_*(\widetilde{\ch}(\mathcal{T})\widetilde{\Td}(T_{\pi})).
\end{equation}
This splits into several terms according to the decomposition of $I \mathscr{C}$ into twisted sectors:
\[I\mathscr{C} = \mathscr{C} \sqcup \bigsqcup_{i=1}^n (\mathscr{S}_i^{(1)} \sqcup \mathscr{S}_i^{(2)}) \sqcup (\mathscr{Z}^{(1)} \sqcup \mathscr{Z}^{(2)}).\]
Here, $\mathscr{S}_i^{(h)}$ is the sector corresponding to the element $h \in \Z_3 = \{0,1,2\}$ of the isotropy group at the $i$th marked point and $\mathscr{Z}^{(h)}$ is the sector corresponding to the element $h$ of the isotropy group at the substratum of nodes.  Applying this decomposition to the right-hand side of (\ref{oGRR}) shows that $\ch(R\pi_*\mathcal{T})$ equals
\[\pi_*(\ch(\mathcal{T})\Td(T_{\pi})) + \sum_{i=1}^n \sum_{\ell}^2 \pi_*(\widetilde{\ch}(\mathcal{T})\widetilde{\Td}(T_{\pi})|_{\mathscr{S}_i^{(\ell)}}) + \sum_{\ell=1}^2 \pi_*(\widetilde{\ch}(\mathcal{T}) \widetilde{\Td}(T_{\pi})|_{\mathscr{Z}^{(\ell)}}).\]

The contribution from the nontwisted sector can be computed via a computation nearly identical to that of Proposition 1.6.3 of \cite{Coates}; the result is:
\[\pi_*\left(\ch(\mathcal{T}) \left( \Td^{\vee}(\overline{L}_{n+1}) - \sum_{i=1}^n s_{i*} \left[ \frac{\Td^{\vee}(N_i^{\vee})}{c_1(N_i^{\vee})}\right]_+ +\right.\right.\hspace{3cm}\]
\[\left.\left.\hspace{5cm}\iota_* \left[\frac{1}{\psi_+ \psi_-}\left(\frac{\Td^{\vee}(L_+)}{\psi_+} + \frac{\Td^{\vee}(L_-)}{\psi_-}\right)\right]_+\right)\right)_k.\]
We have identified the universal family with $\Mt^3_{g,n+1}(\P^1, d)'$, in which the prime indicates that the last marked point has multiplicity $1$.  In the second term, $s_i$ denotes the inclusion of the divisor $\Delta_i$ of the $i$th marked point and $N_i$ denotes the normal bundle of $\Delta_i$ in $\mathscr{C}$.  In the third term, $\iota: Z' \rightarrow \mathscr{C}$ is the composition of the inclusion $i: Z \rightarrow \mathscr{C}$ of the singular locus with the double cover $\gamma: Z' \rightarrow Z$ consisting of choices of a branch at each node; also, $L_{\pm}$ are the cotangent line bundles to the two branches of a node and $\psi_{\pm}$ are the first Chern classes of these line bundles.

Accordingly, we can split $\ch_k(R\pi_*\mathcal{T})$ into a codimension-$0$, codimension-$1$, and codimension-$2$ term, and we compute each separately.

\subsubsection{Codimension 0}

Since $\mathcal{T}^{\otimes 3} \cong \omega_{\log} \otimes f^*\O(-1)$, we have
\[\ch(\mathcal{T}) = \exp(\sfrac{K}{3}) \exp(-\sfrac{f^*H}{3}),\]
where $K = c_1(\omega_{\log})$.  Thus, the codimension $0$ term of $\ch(R\pi_*\mathcal{T})$ is
\[\pi_*(\exp(\sfrac{K}{3})\exp(-\sfrac{f^*H}{3})\Td^{\vee}(\overline{L}_{n+1})).\]
The contribution from the codimension 0 term to (\ref{deriv}), then, is $\mathcal{D}_{tw}$ times the following, in which the superscript $\bullet$ denotes invariants in which the last marked point has multiplicity $1$:
\begin{align*}
&\sum_{g,n,d} \frac{Q^d\hbar^{g-1}}{n!}\left\langle \t, \ldots ; \pi_*\left(\exp(\sfrac{K}{3})\exp(-\sfrac{f^*H}{3})\Td^{\vee}(\overline{L}_{n+1})\right)_{k+1}\; c(R\pi_*\mathcal{T})\right\rangle_{g,n,d} \\
=&\sum_{g,n,d} \frac{Q^d\hbar^{g-1}}{n!} \left\langle \pi^*\t, \ldots, \left(\exp(\sfrac{K}{3})\exp(-\sfrac{H}{3})\Td^{\vee}(\overline{L}_{n+1})\right)_{k+1}\; c(R\pi_*\mathcal{T})\right\rangle^{\bullet}_{g,n+1,d}\\
=&\sum_{g,n,d}  \frac{Q^d\hbar^{g-1}}{n!} \left\langle \t - \sigma_{1*}\left[\textstyle\frac{\t}{\psibar}\right]_+, \ldots, \left(\exp(\sfrac{K}{3})\exp(-\sfrac{H}{3})\Td^{\vee}(\overline{L}_{n+1})\right)_{k+1}; c(R\pi_*\mathcal{T})\right\rangle^{\bullet}_{g,n+1,d}.
\end{align*}
Now, under the identification of the universal family with $\Mt^3_{g,n}(\P^1, d)'$, $K$ is identified with $\psibar_{n+1}$.  Furthermore, $\psibar_{n+1}$ vanishes on the image of each $\sigma_{i*}$ with $1 \leq i \leq n$, so the above is equal to
\[\sum_{g,n,d}  \frac{Q^d\hbar^{g-1}}{n!} (\exp(\sfrac{\psibar_{n+1}}{3})\exp(-\sfrac{H}{3})\Td^{\vee}(\overline{L}_{n+1}))_{k+1}; c(R\pi_*\mathcal{T})\rangle^{\bullet}_{g,n+1,d}\]
\[\hspace{1cm} - \sum_{g,n,d}  \frac{Q^d\hbar^{g-1}}{(n-1)!} \langle \sigma_{1*} \left[ \textstyle\frac{\t}{\psibar}\right]_+, \ldots, (\exp(-\sfrac{H}{3})_{k+1}; c(R\pi_*\mathcal{T})\rangle_{g,n+1,d} \]
\begin{align*}
&= \sum_{g,n,d} \frac{Q^d\hbar^{g-1}}{(n-1)!} \left\langle \t, \ldots, \left(\exp(\sfrac{\psibar}{3})\exp(-\sfrac{H}{3})\Td^{\vee}(\overline{L}_n)\right)_{k+1}; c(R\pi_*\mathcal{T})\right\rangle^{\bullet}_{g,n,d} \\
&\hspace{1cm} - \sum_{g,n,d} \frac{Q^d\hbar^{g-1}}{(n-1)!} \left\langle \t, \ldots, \t, \exp(-\sfrac{H}{3})_{k+1} \left[\textstyle\frac{\t(\psibar)}{\psibar}\right]_+; c(R\pi_*\mathcal{T})\right\rangle_{g,n,d}\\
&\hspace{1cm} - \frac{1}{2\hbar}\left\langle \t, \t, (\exp(\sfrac{\psibar_3}{3}) \exp(-\sfrac{H}{3}) \Td^{\vee}(\overline{L}_3))_{k+1}; c(R\pi_*\mathcal{T})\right\rangle^{\bullet}_{0,3,0}\\
&\hspace{1cm}-\left\langle (\exp(\sfrac{\psibar_1}{2}))\exp(-\sfrac{H}{3})\Td^{\vee}(\overline{L}_1))_{k+1}; c(R\pi_*\mathcal{T})\right\rangle^{\bullet}_{1,1,0} .
\end{align*}
The last two summands are known respectively as the genus-zero and the genus-one exceptional terms.  Since $\psibar_3$ vanishes on $\Mt_{0,3}^3(\P^1, 0)$, the genus zero exceptional term equals
\[-\frac{1}{2\hbar} (\exp(-\sfrac{H}{3})_{k+1} \mathbf{q}, \mathbf{q})_{tw}.\]
The rank of $R\pi_*\mathcal{T}$ is zero on $\Mt_{1,1}^3(\P^1, 0)$, so the genus-one exceptional term does not depend on $s_k$.  It is easily computed, but it will yield only a scalar factor and hence does not affect our present computation.

\subsubsection{Codimension 1}

Since $K$ vanishes on the image of $\sigma_{i*}$ for all $i$, we have $\ch(\left.\mathcal{T}\right|_{\Delta_i}) = \exp(-f^*H/3)$.  Thus, the untwisted contribution to $\ch_k(R\pi_*\mathcal{T})$ from the $i$th marked point is
\[-\pi_*\left(\exp(-\sfrac{f^*H}{3})\;s_{i*}\left[\textstyle\frac{\Td^{\vee}(N_i^{\vee})}{c_1(N_i^{\vee})}\right]_+\right)_k = -\pi_*s_{i*}\left(\exp(-\sfrac{f^*H}{3})\;\left[\textstyle\frac{\Td^{\vee}(N_i^{\vee}}{c_1(N_i^{\vee})}\right]_+\right)_k.\]
If $\sigma_i: \Mt^3_{g,n}(\P^1, d) \rightarrow \Delta_i$ is the $i$th section, then we have $\sigma_{i*} \sigma_i^* = \id$ if the marked point is broad and $\sigma_{i*}\sigma_i^* = 3 \cdot \id$ if the marked point is narrow.  Also, we have $f \circ \sigma_i = \ev_i$, and Lemma 7.3.6 of \cite{Tseng} shows that $\sigma_i^* N_i^{\vee} = L_i$.  Since $\ev_i^*$ is zero away from the summand $H^*_{m_i}(\P^1) \otimes R$ where $m_i$ is the multiplicity of the $i$th marked point, the above can be rewritten as
\[-\frac{1}{r_i}\exp\left(-\frac{H^{(m_i)}}{3}\right)\left[\frac{\Td^{\vee}(L_i)}{\psi_i}\right]_+,\]
where $r_i$ is $1$ if the marked point is broad and $3$ if it is narrow.  Note that the evaluation map in this expression has been suppressed as it will appear as an insertion in twisted invariants.

If the marked point is narrow, there are also twisted sectors, which together contribute
\[\sum_{m=1}^2\pi_*(\widetilde{\ch}(\mathcal{T})\widetilde{\Td}(T_{\pi})|_{\mathscr{S}_i^{(m)}}) = \sum_{m=1}^2 \pi_*s_{i*}\left(\frac{\sum_{0 \leq \ell \leq 1} e^{2\pi i \frac{m \ell}{3}} \ch(\mathcal{T}^{(\ell)}|_{\Delta_i})}{1 - e^{2\pi i \frac{-m}{3}}\ch(N_i^{\vee})}\right),\]
where $\mathcal{T}^{(\ell)}$ is the subbundle of $\mathcal{T}$ in which the isotropy group acts by $e^{2\pi i \frac{\ell}{3}}$.  This is either all of $\mathcal{T}$ or is rank zero, depending on whether $\ell = m_i$, so we can write the above as
\[\frac{1}{3} \exp\left(-\sfrac{H^{(m_i)}}{3}\right) \sum_{1 \leq m \leq 2} \frac{e^{2\pi i \frac{mm_i}{3}}}{1 - e^{2\pi i \frac{-m}{3}}e^{\psi_i}},\]
where we have used $\sigma_i$ as above and again suppressed the evaluation.  It is straightforward to check (see Section 7.3.5 of \cite{Tseng}) that for each $\ell$,
\begin{equation}
\label{identity}
\sum_{1 \leq m \leq 2} \frac{\zeta^{m\ell}}{1 - \zeta^{-m}e^{\psi_i}} = \frac{3 e^{\ell\psi_i}}{1 - e^{3\psi_i}} - \frac{1}{1 - e^{\psi_i}},
\end{equation}
where $\zeta = e^{\frac{2\pi i}{3}}$.  Applying this to the above twisted codimension-$1$ contribution and adding it to the untwisted part, we obtain
\[- \sum_{m \geq 1} \frac{\exp(-\frac{H^{(m_i)}}{3}) B_m(\frac{m_i}{3})}{m!} \psibar_i^{m-1},\]
which is also the total contribution from a broad marked point.  In other words, if $A_m$ is the operator on $H_{tw}$ given by
\[A_m = \bigoplus_{\ell=0}^2 \exp(-\sfrac{H^{(\ell)}}{3}) B_m(\sfrac{\ell}{3}),\]
then the total codimension-$1$ contribution to $\d \mathcal{D}_{tw}/\d s_k$ in either the broad or narrow case is
\[- \sum_{g,n,d} \frac{Q^d \hbar^{g-1}}{(n-1)!} \left\langle \left( \sum_{m \geq 1} \frac{A_m}{m!}\psibar^{m-1}\right)_k \t, \ldots, \t; c(R\pi_*\mathcal{T})\right\rangle_{g,n,d} \mathcal{D}_{tw}.\]

\subsubsection{Codimension 2}

The same exact proof as in \cite{Coates} shows that the untwisted codimension-$2$ contribution to $\ch_k(R\pi_*\mathcal{T})$ can be expressed as
\[\frac{1}{2}\pi_*\iota_*\left(\frac{\ch(\mathcal{T}|_Z)}{\psi_+ + \psi_-}\left(\frac{1}{e^{\psi_+}-1} - \frac{1}{\psi_+} + \frac{1}{2} + \frac{1}{e^{\psi_-} -1} - \frac{1}{\psi_-} + \frac{1}{2}\right)\right).\]

To determine the twisted part, we must calculate the invariant and moving parts of $\iota^*T_{\pi}$.  These can be computed by pulling back the Koszul resolution of the normal bundle of $Z$ in $\mathscr{C}$ to the double cover $Z'$ (Section 7.3.7 of \cite{Tseng}), yielding the exact sequence
\begin{equation}
\label{koszul}
0 \rightarrow L_+ \otimes L_- \rightarrow L_+ \oplus L_- \rightarrow \iota^*T_{\pi} \rightarrow \O_{Z'} \rightarrow \O_{Z'} \rightarrow 0.
\end{equation}
Since the isotropy group acts by $-1$ on both $L_+$ and $L_-$, it acts trivially on their tensor product and nontrivially on their direct sum.  Thus, in $K$-theory we have
\[\iota^*T_{\pi}^{\text{inv}} = -(L_+ \otimes L_-)^{\vee}\]
and
\[\iota^*T_{\pi}^{\text{mov}} = L_+^{\vee} \oplus L_-^{\vee}.\]
By oGRR, then, we compute the twisted codimension-$2$ contribution to $\ch_k(R\pi_* \mathcal{T})$ to be the degree-$k$ part of the following:
\begin{align*}
&\sum_{m=1}^2 \pi_*(\widetilde{\ch}(\mathcal{T}) \widetilde{\Td}(T_{\pi})|_{\mathscr{Z}^{(m)}})\\
=& \frac{1}{2} \sum_{m=1}^2 \pi_* \iota_*\left( e^{2\pi i \frac{m m_{\text{node}}}{3}} \frac{\exp(-\frac{H^{(m_{\text{node}})}}{3})}{\psi_+ + \psi_-}\frac{e^{\psi_+ + \psi_-} -1}{(1-\zeta^{-m} e^{\psi_+})(1 - \zeta^m e^{\psi_+})}\right)\\
=&\frac{1}{2} \sum_{m=1}^2 \pi_*\iota_*\left( \frac{\exp(-\frac{H^{(m_{\text{node}})}}{3})}{\psi_+ + \psi_-} \left( \zeta^{mm_{\text{node}}} + \frac{\zeta^{mm_{\text{node}}}}{\zeta^{-m} e^{\psi_+} -1} + \frac{\zeta^{mm_{\text{node}}}}{\zeta^m e^{\psi_-} -1}\right)\right)
\end{align*}
Here, $m_{\text{node}}$ is the locally constant function on $Z'$ giving the action of the isotropy group at the node on $\mathcal{T}$.  The identity (\ref{identity}) can again be applied to simplify this expression; if $m_{\text{node}} \neq 0$, then we obtain
\[\frac{1}{2}\pi_* \iota_*\left(\frac{\exp(-\frac{H^{(m_{\text{node}})}}{3})}{\psi_+ + \psi_-}\left(-1 +  \frac{3e^{m_{\text{node}}}}{e^{3\psi_-} -1} - \frac{1}{\psi_+} + \frac{3e^{(3-m_{\text{node}})\psi_-}}{e^{3\psi_-} -1} - \frac{1}{\psi_-}\right)\right),\]
which when added to the untwisted codimension-$2$ contribution is
\[\frac{3}{2} \pi_* \iota_*\left( \frac{\exp(-\frac{H^{(m_{\text{node}})}}{3})}{\psi_+ + \psi_-} \left( \sum_{m \geq 2} \frac{B_m(\frac{m_{\text{node}}}{3})}{m!} \psibar_+^{m-1} + \frac{B_m(1 - \frac{m_{\text{node}}}{3})}{m!}\psibar_-^{m-1}\right)\right).\]
In fact, the same holds, via a slightly different computation, when $m_{\text{node}} = 0$.

Adding this to the untwisted part and using the identity $B_m(1-x) = (-1)^mB_m(x)$, one finds that the codimension-$2$ contribution to $\d \mathcal{D}_{tw}/\d s_k$ is $\mathcal{D}_{tw}$ times
\begin{equation}
\label{decomposition}
\frac{1}{2} \sum_{g,n,d} \frac{Q^d\hbar^{g-1}}{n!} \left\langle \t, \ldots; \pi_*\iota_*\left[\sum_{m \geq 2}\textstyle \frac{3r_{\text{node}}A_m}{m!} \frac{\psibar_+^{m-1} + (-1)^m\psibar_-^{m-1}}{\psibar_+ + \psibar_-}\right]_{k-1} \right\rangle^{tw}_{g,n,d},
\end{equation}
in which $r_{\text{node}}$ is $1$ if the node is broad and $3$ if it is narrow.

The idea at this point is to apply the same argument as in Theorem 1.6.4 of \cite{Coates} (see the heading ``Codimension-2 terms") to decompose (\ref{decomposition}) into a sum over the moduli spaces corresponding to the two sides of the node.  It is important to notice, however, that the relevant decomposition property in this setting is slightly different.  Namely, if $\widetilde{D}$ denotes the locus in $\Mt_{g,n}^3(\P^1, d)$ of curves with a separating node in which the two branches have genera $g_i$, $n_i$ marked points, and degrees $d_i$ (for $i=1,2$), then
\[3r_{\text{node}} \left( \frac{3}{\deg(\rho)}\int_{\widetilde{D}} \ev_1^*(\phi_1)\psi_1^{a_1} \cdots \ev_n^*(\phi_n)\psi_n^{a_n} c(R\pi_*\mathcal{T}) \right)\]
\[=\left( \frac{3}{\deg(\rho)}\int_{\Mt_{g_1, n_1+1}^3(\P^1, d_1)} \cdots c(R\pi_*\mathcal{T}) \right)\left( \frac{3}{\deg(\rho)}\int_{\Mt_{g_2, n_2+1}^3(\P^1, d_2)} \cdots c(R\pi_*\mathcal{T})\right),\]
where the integrands on the right-hand side depend on which marked points lie on which components in $\widetilde{D}$ and in all cases the integral is against the pullback of the virtual class under $\rho$.  The proof of this equality is an application of the projection formula, using the fact that if $\rho_D = \left.\rho\right|_{\widetilde{D}}$, then in the case where the node is narrow one has $\deg(\rho_D) = \frac{1}{3} \deg(\rho)$ due to the presence of an additional ``ghost" automorphism acting locally around the node as $(x,y) \mapsto (\zeta x, y)$.  An analogous computation shows the decomposition property for nonseparating nodes.

In particular, the factor of $3r_{\text{node}}$ appearing in (\ref{decomposition}) also appears in the decomposition property for twisted correlators, so (\ref{decomposition}) can be expressed as
\begin{align*}
&\frac{1}{2} \sum_{\substack{g_1, g_2\\ n_1, n_2\\ d_1, d_2}} \frac{Q^{d_1 + d_2}\hbar^{g_1 + g_2 -1}}{n_1!n_2!} \sum_{r, s, \alpha, \beta} \left \langle \t, \ldots, \t, q_r^{\alpha}\phi_{\alpha} \psibar_+^r; c(R \pi_*\mathcal{T})\right\rangle_{g_1, n_1+1, d_1} \times\\
&\hspace{2cm} \left\langle q_s^{\beta}\phi_{\beta}\psibar_-^s, \t, \ldots, \t; c(R\pi_*\mathcal{T})\right\rangle_{g_2, n_2 + 1, d_2} \mathcal{D}_{tw}\\
+&\frac{1}{2}\sum_{g,n,d} \frac{Q^d \hbar^{g-1}}{n!} \sum_{r,s,\alpha, \beta} \left \langle \t, \ldots, \t, q_r^{\alpha}\phi_{\alpha} \psibar_+^r, q_s^{\beta}\phi_{\beta}\psibar_-^s; c(R\pi_*\mathcal{T}) \right\rangle_{g-1,n,d} \mathcal{D}_{tw},
\end{align*}
where the $q$'s are determined by the requirement that $\displaystyle\sum_{r,s,\alpha, \beta} q_r^{\alpha}\phi_{\alpha}\psibar_+^r\otimes q_s^{\beta} \phi_{\beta} \psibar_-^s $ equals
\[\left( \sum_{m \geq 2} \frac{A_m}{m!} \frac{\psibar_+^{m-1} + (-1)^m \psibar_-^{m-1}}{\psibar_+ + \psibar_-}\right)_{k-1} \wedge (g^{\alpha\beta}\phi_{\alpha} \otimes \phi_{\beta})\]
and $g^{\alpha \beta}$ is the inverse of the matrix for the twisted Poincar\'e pairing.

By Appendix C of \cite{Tseng}, this equals
\[\frac{\hbar}{2} (\d \otimes_{C_k} \d) \mathcal{D}_{tw}\]
for
\[C_k = \bigoplus_{\ell=0}^2 \sum_{m \geq 1} \frac{B_m(\frac{\ell}{3})}{m!} \exp(-\sfrac{H^{(\ell)}}{3})_{k+1-m} z^{m-1}.\]

\subsubsection{Putting everything together}

The sum of the codimension-$1$ and nonexceptional codimension-$0$ contributions is
\[\displaystyle\sum_{g,n,d} \frac{Q^d\hbar^{g-1}}{(n-1)!} \langle \t, \ldots, \t, (\exp(\sfrac{\psibar_n}{3})\exp(-\sfrac{H}{3})\Td^{\vee}(L_n))_{k+1}; c(R\pi_*\mathcal{T})\rangle^{\bullet}_{g,n,d} \mathcal{D}_{tw}\]
\begin{equation}
\label{codim0and1}
-\displaystyle\sum_{g,n,d} \frac{Q^d\hbar^{g-1}}{(n-1)!} \langle \t, \ldots, \t, \exp(-\sfrac{H}{3})_{k+1} \left[\textstyle\frac{\t(\psibar)}{\psibar}\right]_+; c(R\pi_*\mathcal{T})\rangle_{g,n,d} \mathcal{D}_{tw}
\end{equation}
\[\hspace{1cm}-\displaystyle\sum_{g,n,d} \frac{Q^d\hbar^{g-1}}{(n-1)!} \left\langle \left(\sum_{m \geq 1} \frac{A_m}{m!}\; \psibar^{m-1}\right)_k \t, \ldots, \t; c(R\pi_*\mathcal{T})\right\rangle_{g,n,d} \mathcal{D}_{tw}.\]
Using that
\[\exp(-\sfrac{H}{3})_{k+1}\left[\sfrac{\t(\psibar)}{\psibar}\right]_+ = \displaystyle\bigoplus_{\ell=0}^2 \exp(-\sfrac{H^{(\ell)}}{3})_{k+1} \left(\lfrac{\t(\psibar)- t_0}{\psibar}\right)\]
and
\[\sum_{m \geq 1} \frac{A_m}{m!} z^{m-1} = \bigoplus_{\ell=0}^2 \exp(-\sfrac{H^{(\ell)}}{3}) \left( \lfrac{e^{\frac{\ell}{3}z}}{e^z -1} - \frac{1}{z}\right),\]
we find that the sum of the second two terms in (\ref{codim0and1}) is $\mathcal{D}_{tw}$ times
\[-\sum_{g,n,d} \frac{Q^d\hbar^{g-1}}{(n-1)!} \left\langle \left[ \left( \frac{\sum_{0 \leq \ell \leq 2} \exp(-\frac{H^{(\ell)}}{3}) e^{\frac{\ell}{3}\psibar}}{e^{\psibar}-1}\right)_k \t(\psibar)\right]_+, \ldots ; c(R\pi_*\mathcal{T})\right\rangle_{g,n,d}.\]
Also, keeping in mind that $1 \in H^*_1(\P^1)$, we find that the contribution from the remaining codimension-$0$ nonexceptional term is equal to
\begin{align*}
(\exp(\psibar/3)\exp(-H/3)\Td^{\vee}(\overline{L_n}))_{k+1} &= \left(\frac{\exp(-\frac{H}{3})e^{\frac{1}{3}\psibar}}{e^{\psibar}-1}\psibar\right)_{k+1}\\
&= \left[\left(\frac{\exp(-\frac{H}{3})e^{\frac{1}{3}\psibar}}{e^{\psibar}-1}\right)_k 1\psibar\right]_+\\
&=\left[ \left( \frac{ \sum_{0 \leq \ell \leq 2} \exp(- \frac{H^{(\ell)}}{3}) e^{\frac{\ell}{3} \psibar}}{e^{\psibar} -1}\right)_k 1\psibar\right]_+.
\end{align*}
Therefore, the sum of the codimension-$1$ and nonexceptional codimension-$0$ terms is $\mathcal{D}_{tw}$ times
\[-\sum_{g,n,d} \frac{Q^d\hbar^{g-1}}{(n-1)!} \left\langle \left[ \left(\frac{\sum_{0 \leq \ell \leq 2} \exp(-\frac{H^{(\ell)}}{3}) e^{\frac{\ell}{3}\psibar}}{e^{\psibar} -1}\right)_k \mathbf{q}(\psibar)\right]_+ , \ldots; c(R\pi_*\mathcal{T})\right\rangle_{g,n,d}\]
and the computations in Example 1.3.3.1 of \cite{Coates} shows that this equals $-\d_{C_k}\mathcal{D}_{tw}$ for $C_k$ as above.

Combining everything and using the explicit description of quantized operators in Section 1.3.3 of \cite{Coates}, we have proved that
\[\frac{\d \mathcal{D}_{tw}}{\d s_k} = \frac{1}{2\hbar} \Omega_{tw}((C_k\mathbf{q})(-z), \mathbf{q}(z)) - \d_{C_k}\mathcal{D}_{tw} + \frac{\hbar}{2}(\d \otimes_{C_k} \d) \mathcal{D}_{tw} = \widehat{C_k}\mathcal{D}_{tw},\]
which is part (a) of the proposition.

The proof of part (b) is nearly identical and somewhat simpler, since there is only one nontrivial twisted sector associated to each marked point and to the divisor of nodes, so we omit it.
\end{proof}

\subsection{Landau-Ginzburg $I$-function}

As in \cite{CR}, \cite{Coates}, and \cite{CCIT}, one can define a certain hypergeometric modification $I_{tw}$ of the untwisted $J$-function in such a way that the family $\Delta^{-1}I_{tw}(t, -z)$ lies on the untwisted Lagrangian cone $\mathcal{L}_{un}$; in light of the above, it follows that $I_{tw}(t,-z) \in \mathcal{L}_{tw}$.  When we take a nonequivariant limit $\lambda \rightarrow 0$ and set the parameters $s_k$ as in (\ref{sk}), we will obtain a family lying on $\mathcal{L}_{hyb}$, and in fact, this family will determine the entire cone just as the hybrid $J$-function does.

As usual, we will define $I_{tw}$ only in the case of the cubic singularities, commenting only briefly on how to apply the same procedure to define $I_{tw}$ in the other case.

\subsubsection{Setup in cubic case}

First, decompose $J_{un}$ according to topological types, as in \cite{CCIT}.  The {\it topological type} of an element of some $\Mt^3_{g,n}(\P^1, d)$ is the triple $\Theta = (g,d,\mathbf{i})$, where $g$ is the genus of the source curve, $d$ is the degree of the map, and $\mathbf{i} = (i_1, \ldots, i_n)$ gives the multiplicities of the line bundle at each of the marked points.  Let $J_{\Theta}$ be the contribution to $J_{un}$ from invariants of topological type $\Theta$, and write
\[J_{un}(t, z) = \sum_{\Theta} J_{\Theta}(t, z),\]
where the sum is over all topological types.\footnote{The $z+t$ term in $J_{un}(t,z)$ should be understood as contributing to the unstable topological types corresponding to $(g,n,d) = (0,1,0)$ and $(0,2,0)$.}

Let us also fix some notation, again mimicking \cite{CCIT}.  Set
\[\mathbf{s}(x) = \sum_{k \geq 0} s_k \frac{x^k}{k!}\]
for any $x \in \mathcal{V}_{tw}^+$.  For a multiplicity $h \in \{0,1,2\}$, let
\[D_{(h)} = \sum_{\alpha=0}^1 t^{\alpha}_{0, (h)} \frac{\d}{\d t^{\alpha}_{0, (h)}}\]
denote the dilation vector field on $H^*_{h}(\P^1)$, where for $\t = t_0 + t_1z + t_2z^2 + \cdots \in H_{GW}[[z]]$ we write
\[t_i = \sum_{\substack{0 \leq \alpha \leq 1\\ 0 \leq h \leq 2}} t^{\alpha}_{i, (h)} \phi_i^{(h)}.\]
with $\{\phi_i\}$ denoting a basis for $H^*_h(\P^1)$.  Also, set
\[G_y(x,z) = \sum_{k,m \geq 0} s_{k+m-1} \frac{B_m(y)}{m!} \frac{x^k}{k!} z^{m-1}\]
for $y \in \Q$ and $x \in H_{tw}$, where $z$ denotes the variable in $\mathcal{V}_{tw}$, as usual.

For each topological type $\Theta$, let $\overline{i_n}$ be the multiplicity that is equal modulo $3$ to $-i_n$.  Set
\[N_{\Theta} = \frac{-2 + n - d - \sum_{j=1}^{n-1}i_j}{3} + \frac{\overline{i_n}}{3}.\]
Note that this is an integer, since it equals either
\[\frac{-2 + n - d - \sum_{j=1}^n i_j}{3} = \deg(|L|)\]
or $\deg(|L|) + 1$ depending on whether $i_n$ is zero or nonzero.  Thus, we can set
\[M_{\Theta} = \frac{ \displaystyle\prod_{- \infty < m \leq N_{\Theta}} \exp\left(\mathbf{s}(-\sfrac{H^{(\overline{i_n})}}{3} + (m - \sfrac{\overline{i_n}}{3})z)\right)}{ \displaystyle \prod_{-\infty < m \leq 0} \exp\left(\mathbf{s}(-\sfrac{H^{(\overline{i_n})}}{3} + (m - \sfrac{\overline{i_n}}{3})z)\right)}\]
Note that these definitions of $N_{\Theta}$ and $M_{\Theta}$ are direct generalizations of those appearing in \cite{CCIT}, and the same proof shows that the properties in Lemma 4.5 and equations (12) and (13) of that paper still hold.

\subsubsection{Quadric case}

The definitions of $\mathbf{s}(x)$ and $G_y(x,z)$ remain unchanged in the case of the quadric singularities, while the dilation vector field on $H^*_h(\P^3)$ changes only in that the summation runs over a basis for $H^*(\P^3)$, so $0 \leq \alpha \leq 3$.  As for $N_{\Theta}$, we should now take $\overline{i_n}$ to be equal to $-i_n$ modulo $2$, which is the same as setting $\overline{i_n} = i_n$.  The resulting definition is:
\[N_{\Theta} = \frac{-2 + n - d - \sum_{j=1}^{n-1} i_j}{2} + \frac{i_n}{2}.\]
Similarly,
\[M_{\Theta} = \frac{ \displaystyle\prod_{- \infty < m \leq N_{\Theta}} \exp\left(\mathbf{s}(-\sfrac{H^{(i_n)}}{2} + (m - \sfrac{i_n}{2})z)\right)}{ \displaystyle \prod_{-\infty < m \leq 0} \exp\left(\mathbf{s}(-\sfrac{H^{(i_n)}}{2} + (m - \sfrac{i_n}{2})z)\right)}\]
Once again, the necessary properties of these expressions follow direct from the analogues in \cite{CCIT}.

\subsubsection{Twisted $I$-function}

In either of the two cases under consideration, define
\[I_{tw}(\t, z) = \sum_{\Theta} M_{\Theta}(z) \; J_{\Theta}(\t, z).\]
The hybrid $I$-function will be defined by putting $s_k$ to the values in (\ref{sk}), taking $\lambda \rightarrow 0$, specializing to multiplicity-$1$ divisor insertions with no $\psi$ classes, and multiplying by a factor.

\begin{theorem}
\label{Ifunction}
\begin{enumerate}[(a)]
\item For the cubic singularity, define
\[I_{hyb}(t, z) = \sum_{\substack{d \geq 0\\ d \not \equiv -1 \mod 3}} \frac{ze^{(d +1+ \frac{H^{(d+1)}}{z})t}}{3^{6\lfloor \frac{d}{3}\rfloor}} \; \frac{ \displaystyle\prod_{\substack{1 \leq b \leq d\\ b \equiv d+1 \mod 3}} (H^{(d+1)}+bz)^{4}}{ \displaystyle\prod_{\substack{1 \leq b \leq d\\ b \not \equiv d+1 \mod 3}} (H^{(d+1)}+bz)^{2}},\]
where $t = t + 0z + 0z^2 + \cdots \in \mathcal{V}_{hyb}^+$ and $t \in H^2_1(\P^1)$.  Then the family $I_{hyb}(t, -z)$ of elements of $\mathcal{V}_{hyb}$ lies on the Lagrangian cone $\mathcal{L}_{hyb}$.

\item For the quadric singularity, define
\[I_{hyb}(t, z) = \sum_{\substack{d \geq 0\\ d \not \equiv -1 \mod 2}} \frac{ze^{(d +1+ \frac{H^{(d+1)}}{z})t}}{2^{8\lfloor \frac{d}{2}\rfloor}} \; \frac{ \displaystyle\prod_{\substack{1 \leq b \leq d\\ b \equiv d+1 \mod 2}} (H^{(d+1)}+bz)^{4}}{ \displaystyle\prod_{\substack{1 \leq b \leq d\\ b \not \equiv d+1 \mod 2}} (H^{(d+1)}+bz)^{4}},\]
where $t \in H^2_1(\P^3)$.  Then the family $I_{hyb}(t, -z)$ of elements of $\mathcal{V}_{hyb}$ lies on the Lagrangian cone $\mathcal{L}_{hyb}$.
\end{enumerate}

\begin{remark}
These $I$-functions have expressions in terms of the $\Gamma$ function, which can be useful for computations-- see the appendix.
\end{remark}

\begin{proof}
The proof mimics that of Theorem 4.6 of \cite{CCIT}.  We will begin by proving that $I_{tw}(\t, -z)$ lies on $\mathcal{L}_{tw}$ for the cubic singularity, and then show how to obtain $I_{hyb}$ from $I_{tw}$.  As usual, everything we say will carry over to the quadric case with only minor modifications, so we omit the proof.

Using equations (12) and (13) of \cite{CCIT}, it is easy to check that
\begin{align*}
M_{\Theta}(-z) &= \exp\left( G_0\left(-\sfrac{H^{(\overline{i_n})}}{3} + \sfrac{\overline{i_n}}{3}z, z\right) - G_0\left(- \sfrac{H^{(\overline{i_n})}}{3} + (\sfrac{\overline{i_n}}{3} - N_{\Theta})z, z\right)\right)\\
&= \exp\left( G_{\sfrac{\overline{i_n}}{3}} \left(-\sfrac{H^{(\overline{i_n})}}{3}, z\right) - G_0\left(- \sfrac{H^{(\overline{i_n})}}{3} + (\frac{\overline{i_n}}{3} - N_{\Theta})z, z\right) \right).
\end{align*}
Furthermore,
\[\Delta = \bigoplus_{\ell=0}^{3} \exp\left(G_{\frac{\ell}{3}}\left( - \sfrac{H^{(\ell)}}{3}, z\right)\right).\]

Given that $\Delta(\mathcal{L}_{un}) = \mathcal{L}_{tw}$, the desired statement is equivalent to the statement $\Delta^{-1}I_{tw}(\t, -z) \in \mathcal{L}_{un}$.  Using Lemma 4.5(1) of \cite{CCIT} and the above expression for $M_{\Theta}(-z)$, this is equivalent to
\[\sum_{\Theta} \exp\left( - G_{\frac{1}{3}}\left( - \sfrac{H^{(\overline{i_n})}}{3} + \sfrac{d}{3}z - \sfrac{\sum_{j=1}^{n-1}(1-i_j)}{3}z, x\right)\right)J_{\Theta}(\t, -z) \in \mathcal{L}_{un}.\]
Now, we can write
\[\frac{\sum_{j=1}^{n-1}(1-i_j)}{3} = \frac{n_0}{3} - \frac{n_2}{3},\]
and Lemma 4.5(2) of \cite{CCIT} says that $n_0$ and $n_2$ act on $J_{\Theta}$ in the same way, respectively, as $D_{(0)}$ and $D_{(2)}$.  Furthermore, $-\frac{H^{(\overline{i_n})}}{3} + \frac{d}{3}z$ acts on $J_{\Theta}$ in the same way as does $-z\nabla_{-\frac{H^{(\overline{i_n})}}{3}}$.  So if $D = \frac{1}{3} D_{(0)} - \frac{1}{3} D_{(2)}$, we can re-express the desired statement as
\begin{equation}
\label{Js}
\exp\left(- G_{\frac{1}{3}}\left(z \nabla_{-\frac{H}{3}} - zD, z\right)\right) J_{un}(\t, z) \in \mathcal{L}_{un},
\end{equation}
where $H = H^{(0)} + H^{(1)} + H^{(2)}$.

Denote the expression in (\ref{Js}) by $J_{\mathbf{s}}(\t, -z)$.  To prove that $J_{\mathbf{s}}(\t, -z) \in \mathcal{L}_{un}$ is to show that
\[E_j(J_{\mathbf{s}}(\t, -z)) = 0\]
for all $j$, where $E_j$ are the functions $\mathcal{V}_{un} \rightarrow H_{un}$ given by
\[(\mathbf{p}, \mathbf{q}) \mapsto p_j - \sum_{n,d, \alpha, h} \frac{Q^d}{n!} \langle \t, \ldots, \t, \psi^j \phi_{\alpha}^{(h)} \rangle^{un}_{g,n+1, d} \phi^{\alpha, (h)}.\]
This is proved exactly as in \cite{CCIT}, by induction on the degree of the terms in $E_j(J_{\mathbf{s}})(\t, -z)$ with respect to the variables $s_k$ under the convention that $s_k$ has degree $k+1$.

The terms of degree $0$ vanish, as such terms are constant with respect to the $s_k$ and vanish when all $s_k$ are $0$ because $J_{\mathbf{0}} = J_{un}$.  Assume, then, that $E_j(J_{\mathbf{s}}(\t, -z))$ vanishes up to degree $n$ in the variables $s_k$.  To show that it vanishes up to degree $n+1$, we will show that $\frac{\d}{\d s_i} E_j(J_{\mathbf{s}}(\t, -z))$ vanishes up to degree $n$ for all $i$.  We have
\[\frac{\d}{\d s_i} E_j(J_{\mathbf{s}}(\t, -z)) = d_{J_{\mathbf{s}}(\t, -z)} E_j(z^{-1}P_i J_{\mathbf{s}}(\t, -z)),\]
where
\[P_i = -\sum_{m=0}^{i+1} \frac{1}{m!(i+1-m)!} z^m B_m(0)(-z\Delta_{-\frac{H}{3}} - zD)^{i+1-m}.\]
The inductive hypothesis implies the existence of an element $\widetilde{J}_{\mathbf{s}}(\t, -z) \in \mathcal{L}_{un}$ that agrees with $J_{\mathbf{s}}(\t, -z)$ up to degree $n$ in the $s_k$, and hence satisfies
\[\frac{\d}{\d s_i} E_j J_{\mathbf{s}}(\t, -z) = d_{\widetilde{J}_s(\t, -z)} E_j(z^{-1}P_i\widetilde{J}_{\mathbf{s}}(\t, -z))\]
up to degree $n$ in these variables.  It suffices, then, to show that the right-hand side of this equation is identically zero, or in other words that
\[P_i \widetilde{J}_{\mathbf{s}}(\t, -z) \in z T_{\widetilde{J}_{\mathbf{s}}(\t, -z)} \mathcal{L}_{un} = \mathcal{L}_{un} \cap T_{\widetilde{J}_{\mathbf{s}}(\t, -z)} \mathcal{L}_{un}.\]
Let $T =T_{\widetilde{J}_{\mathbf{s}}(\t, -z)} \mathcal{L}_{un}$.  Breaking $P_i$ up into a sum of terms of the form
\[C z^a (z \nabla_{-\frac{H}{3}})^b (zD)^c\]
for a coefficient $C$ and exponents $a,b,$ and $c$, it suffices to show that $z$, $z \nabla_{-\frac{H}{3}}$, and $zD$ all preserve $zT$.  In the first case, this is because $zT = \mathcal{L}_{un} \cap T \subset T$ and hence $z(zT) \subset zT$.  In the second case, the operator $\nabla_{-\frac{H}{3}}$ is a first-order derivative and hence takes $\mathcal{L}_{un}$ to $T$; it follows that $\nabla_{-\frac{H}{3}}$ takes $zT = \mathcal{L}_{un} \cap T \subset \mathcal{L}_{un}$ to $T$ also, and hence $z\nabla_{-\frac{H}{3}}$ takes $zT$ to $zT$.  The same argument applies to the operator $zD$, so this completes the proof that $I_{tw}(\t, -z) \in \mathcal{L}_{tw}$ in the cubic case.

Now suppose we set $s_k$ as in (\ref{sk}), so that $ c(-V) = e_{\C^*}(V^{\vee})^{6}$, and take a limit $\lambda \rightarrow 0$.  It is easy to check via the Taylor expansion of the natural logarithm that in the cubic case, we get
\[M_{\Theta}(z) = \frac{\displaystyle\prod_{- \infty < m \leq 0}\left(\sfrac{H^{(\overline{i_n})}}{3} + (\sfrac{\overline{i_n}}{3} - m)z\right)^{6}}{\displaystyle\prod_{- \infty < m \leq N_{\Theta}}\left(\sfrac{H^{(\overline{i_n})}}{3} + (\sfrac{\overline{i_n}}{3} - m)z\right)^{6}}.\]
Restrict $\t$ to allow only those insertions in $H^2_1(\P^1)$ with no $\psi$ classes; in this case,
\[N_{\Theta} = \frac{-d-1}{3} + \frac{\overline{i_n}}{3},\]
which is always nonpositive, and $\overline{i_n} \equiv d+1 \mod 3$.  Thus, we obtain
\[M_{\Theta}(z) = \prod_{\substack{0 \leq b < \frac{d+1}{3}\\ \{b\} = \{\frac{d+1}{3}\}}} \left(\frac{H^{(d+1)}}{3} + bz\right)^{6},\]
where we use the convention $H^{(h)} = H^{(h \mod 3)}$ if $h \geq 3$.  Notice that if $d + 1 \equiv 0 \mod 3$, then one of the factors in the above product is $b=0$, in which case the product is $0$ because $H^2 = 0$.  Thus, $M_{\Theta}(z)$ vanishes in these cases.\footnote{In fact, we already knew that this had to be the case, because the fact that $I_{tw}(t,-z) \in \mathcal{L}_{tw}$ implies that $I_{tw}(t, z)$ differs from the small hybrid $J$-function by a change of variables, and the hybrid invariants vanish if any insertion is broad.}

Set $t = t_0 + 0z + 0z^2 + \cdots$.  Since untwisted invariants are essentially Gromov-Witten invariants of $\P^1$, we can compute $J_{\Theta}(t,z)$ explicitly in every case where $\Theta$ corresponds to a nonempty component of the moduli space.  Indeed, Givental's Mirror Theorem for $\P^1$ states that
\[1 + \sum_{d, \alpha} Q^d \left \langle \frac{\phi_{\alpha}}{z-\psi}, 1 \right\rangle_{0,d} \phi^{\alpha} = \sum_d Q^d \frac{1}{((H+z) \cdots (H+dz))^2}.\]
Using the string and divisor equations, then, one can show that
\[\sum_{\Theta \text{ with degree } d} J_{\Theta} = \frac{3ze^{(\frac{H^{(d+1)}}{z} +d)t}}{((H^{(d+1)} + z) \cdots (H^{(d+1)} + dz))^2}.\]
Since all $\Theta$ with the same degree yield the same $M_{\Theta}$, namely
\[M_{\Theta} = \frac{1}{3^{6\lfloor \frac{d}{3}\rfloor}} \prod_{\substack{1 \leq b \leq d\\ b \equiv d + 1 \mod 3}} (H^{(d+1)} + bz)^6,\]
we obtain a formula for $I_{tw}(t,z)$.  Writing $\Theta = (0, d, (1, \ldots, 1))$ (with $k$ $1$'s) and taking $Q \rightarrow 1$ as is done in the Gromov-Witten setting, the formula is
\[I_{tw}(t, z) =\sum_{\substack{d \geq 0\\ d \not \equiv -1 \mod 3}} \frac{3ze^{(\frac{H^{(d+1)}}{z}+d)t}}{3^{6\lfloor \frac{d}{3}\rfloor}} \; \frac{ \displaystyle\prod_{\substack{1 \leq b \leq d\\ b \equiv d+1 \mod 3}} (H^{(d+1)}+bz)^{4}}{ \displaystyle\prod_{\substack{1 \leq b \leq d\\ b \not \equiv d+1 \mod \delta}} (H^{(d+1)}+bz)^{2}}.\]
Multiplying by $\frac{1}{3}e^{t}$, which preserves $\mathcal{L}_{hyb}$ because it is a cone, gives the function $I_{hyb}$ of the statement.  An analogous computation shows that the hybrid $I$-function in the quadric case is as stated.
\end{proof}
\end{theorem}

\subsection{Relating the LG and GW $I$-functions}

Equipped with an explicit expression for the hybrid $I$-functions and having proved that they lie on the respective Lagrangian cones $\mathcal{L}_{tw}$, we are finally ready to prove the main theorem of the paper:

\begin{proof}[Proof of Theorem \ref{LGCY}]
We have shown that $I_{hyb}(t, -z)$ lies on the Lagrangian cone $\mathcal{L}_{hyb}$.  The property (\ref{tangent}) implies that the $J$-function is characterized by the fact that $J_{hyb}(t, -z) \in \mathcal{L}_{hyb}$ together with the first two terms of its expansion in powers of $z$:
\[J_{hyb}(t, -z) = -1^{(1)}z + t + O(z^{-1}).\]
Using the formula for $I_{hyb}(t, z)$, it is not difficult to show that it can be expressed as
\[I_{hyb}(t, z) = \omega_1^{hyb}(t) \cdot 1^{(1)} \cdot z + \omega_2^{hyb}(t)  + O(z^{-1})\]
for $\C$-valued functions $\omega_1^{hyb}$ and $\omega_2^{hyb}$.  These can be calculated explicitly, but the computation is tedious and not strictly necessary to prove the LG/CY correspondence, so we relegate it to the Appendix.

Having obtained such functions $\omega_i^{hyb}$, we have
\[\frac{I_{hyb}(t,-z)}{\omega_1^{hyb}(t)} = -1^{(1)} \cdot z + \frac{\omega_2^{hyb}(t)}{\omega_1^{hyb}(t)} + O(z^{-1}),\]
and this still lies on $\mathcal{L}_{hyb}$ because it is a cone.  So by the uniqueness property of $J_{hyb}$, we have
\begin{equation}
\label{cov}
\frac{I_{hyb}(t, -z)}{\omega_1^{hyb}(t)} = J_{hyb}(t', -z), \text{ where } t' = \frac{\omega_2^{hyb}(t)}{\omega_1^{hyb}(t)}.
\end{equation}
This is the change of variables relating the hybrid $I$-function and $J$-function.

As for the symplectic transformation matching $I_{hyb}$ with the analytic continuation of $I_{GW}$, the comments in the Introduction show that it is sufficient to prove that the hybrid $I$-function assembles the solutions to the Picard-Fuchs equation 
\begin{equation}
\label{PF}
\left[ \left(\psi \frac{\d}{\d \psi}\right)^4 - \psi^{-1} \left(\psi\frac{\d}{\d \psi} -\frac{1}{3}\right)^2\left(\psi\frac{\d}{\d \psi} -\frac{2}{3}\right)^2 \right] F = 0
\end{equation}
for the cubic singularity, where $\psi = e^{3t}$, or 
\[\left[ \left(\psi \frac{\d}{\d \psi}\right)^4 - \psi^{-1} \left(\psi\frac{\d}{\d \psi} -\frac{1}{2}\right)^4 \right] F= 0\]
for the quadric singularity, where $\psi = e^{2t}$.  As usual, we prove only the first of these statements.

Split $I_{hyb}$ into two parts corresponding to the two narrow summands of $H_{tw}$, changing the variable of summation in each:
\begin{align*}
I_{hyb}(t,z) = &\sum_{d \geq 0} \frac{ze^{(3d + 1 + \frac{H^{(1)}}{z})t}}{3^{6d}} \frac{\displaystyle\prod_{\substack{1 \leq b \leq 3d\\ b\equiv 1 \mod 3}} (H^{(1)}+bz)^4}{\displaystyle \prod_{\substack{1 \leq b \leq 3d\\ b \equiv 0,2 \mod 3}} (H^{(1)}+bz)^2}\\
+&\sum_{d \geq 0} \frac{ze^{(3d + 2 + \frac{H^{(2)}}{z})t}}{3^{6d}} \frac{\displaystyle\prod_{\substack{1 \leq b \leq 3d+1\\ b\equiv 2 \mod 3}} (H^{(2)}+bz)^4}{\displaystyle \prod_{\substack{1 \leq b \leq 3d+1\\ b \equiv 0,1 \mod 3}} (H^{(2)}+bz)^2}.
\end{align*}
The claim is that, when we set $\psi = e^{3t}$, each of these summands separately satisfies (\ref{PF}) as a cohomology-valued function.  For the first summand, let $\Psi_d$ be the contribution from $d$:
\[\Psi_d = z\frac{\psi^{d + \frac{1}{3} + \frac{H^{(1)}}{3z}}}{3^{6d}} \frac{\displaystyle\prod_{\substack{1 \leq b \leq 3d\\ b\equiv 1 \mod 3}} (H^{(1)}+bz)^4}{\displaystyle \prod_{\substack{1 \leq b \leq 3d\\ b \equiv 0,2 \mod 3}} (H^{(1)}+bz)^2}.\]
By computing the ratio $\Psi_d/\Psi_{d-1}$, it is easy to check that
\[\left(\sfrac{H^{(d+1)}}{3z} + d - \frac{2}{3}\right)^4\Psi_{d-1} = 3^6\psi^{-1} \left(\sfrac{H^{(d+1)}}{3z} + d\right)^2\left(\sfrac{H^{(d+1)}}{3z} + d - \frac{1}{3}\right)^2 \Psi_d.\]
But the operator $\psi \frac{\d}{\d \psi}$ acts on $\Psi_d$ by multiplication by $\left(\frac{H^{(d+1)}}{3z} + d + \frac{1}{3}\right)$, so the above can be expressed as
\[\left(\psi \frac{\d}{\d \psi}\right)^4 \Psi_{d-1} = 3^6\psi^{-1}\left(\psi \frac{\d}{\d \psi} - \frac{1}{3}\right)^2 \left(\psi \frac{\d}{\d \psi} - \frac{2}{3}\right)^2 \Psi_d.\]
It follows that if one applies the Picard-Fuchs operator in (\ref{PF}) to the first summand of $I_{hyb}(t,z)$, all but possibly the $\Psi_0$ summand will be annihilated.  In fact, though, it is easy to see using the fact that $H^2=0$ that $\Psi_0$ is also killed.  Thus, the Picard-Fuchs equation holds for this summand, and an analogous argument proves the same claim for the second summand.
\end{proof}

\section*{Appendix: Explicit mirror map}

In order to explicitly compute the change of variables (\ref{cov}), it is necessary to find the coefficients of $z^1$ and $z^0$ in $I_{hyb}(t,z)$.  Let us do this first in the cubic case.

Using the identity
\[z^{\ell}\frac{\Gamma(1 + \frac{x}{z} + \ell)}{\Gamma(1 + \frac{x}{z})} = \prod_{k=1}^{\ell} (x + kz),\]
one can rewrite $I_{hyb}$ as
\[z \sum_{\substack{d \geq 0 \\ d \equiv -1 \text{ mod } 3}} e^{(d+1 + \frac{H^{(d+1)}}{z})t} z^{-6\langle\frac{d}{3}\rangle} \frac{\Gamma(\frac{H^{(d+1)}}{3z} + \frac{d}{3} + \frac{1}{3})^6}{\Gamma(\frac{H^{(d+1)}}{3z} + \langle \frac{d}{3} \rangle + \frac{1}{3})^6} \frac{\Gamma(\frac{H^{(d+1)}}{z} + 1)^2}{\Gamma(\frac{H^{(d+1)}}{z} + d + 1)^2}.\]
It is easy to see from here that the only terms that contribute to the coefficient of either $z^1$ or $z^0$ are those with $d \equiv 0 \mod 3$.  In particular, if we expand the function
\[F(\eta) = \sum_{\substack{d \geq 0 \\d \equiv 0 \text{ mod }3}} e^{(d+1+\eta)t} \frac{\Gamma(\frac{\eta}{3} + \frac{d}{3} + \frac{1}{3})^6}{\Gamma(\frac{\eta}{3} + \langle \frac{d}{3}\rangle + \frac{1}{3})^6} \frac{\Gamma(\eta + 1)^2}{\Gamma(\eta + d+1)^2}\]
in powers of $\eta$, then
\[\omega_1^{hyb}(t) = F(0) = \sum_{d \geq 0 } e^{(3d + 1)t} \frac{\Gamma(d + \frac{1}{3})^6}{\Gamma(\frac{1}{3})^6\Gamma(3d + 1)^2}\]
and
\[\omega_2^{hyb}(t) = F'(0) = \sum_{d \geq 0}e^{(3d+1)t} \frac{\Gamma(d + \frac{1}{3})^5}{\Gamma(\frac{1}{3})^6\Gamma(3d+1)^2}\bigg( 2\Gamma'(d + \sfrac{1}{3}) + 2\Gamma(d + \sfrac{1}{3})\psi(1)\]
\[\hspace{4.5cm}- 2\Gamma(d + \sfrac{1}{3})\psi(\frac{1}{3}) - 2\Gamma(d + \sfrac{1}{3})\psi(3d+1) + t \Gamma(d + \sfrac{1}{3})\bigg),\]
where $\psi$ is the digamma function, the logarithmic derivative of $\Gamma$.

The same argument shows that in the quadric case, one has
\[\omega_1^{hyb}(t) = G(0) = \sum_{d \geq 0} e^{(d+1)t} \frac{(2d)!^8 (2d+1)!^4}{4^{8d}d!^8}\]
and
\[\omega_2^{hyb}(t) = G'(0) = \sum_{d \geq 0} e^{(d+1)t}\frac{\Gamma(d + \frac{1}{2})^6}{\Gamma(\frac{1}{2})^8\Gamma(2d+1)^4}\bigg(4\Gamma'(d + \sfrac{1}{2}) + 4\Gamma(d + \sfrac{1}{2})\psi(1)\]
\[\hspace{4.5cm}- 4\Gamma(d + \sfrac{1}{2})\psi(\sfrac{1}{2}) - 4\Gamma(d + \sfrac{1}{2})\psi(2d+1) + t \Gamma(d + \sfrac{1}{2})\bigg),\]
where
\[G(\eta) = \sum_{\substack{d \geq 0\\d \equiv 0 \text{ mod }2}} e^{(d+1 + \eta)t} \frac{\Gamma(\frac{\eta}{2} + \frac{d}{2} + \frac{1}{2})^8}{\Gamma(\frac{\eta}{2} + \frac{1}{2})^8} \frac{\Gamma(\eta + 1)^4}{\Gamma(\eta + d + 1)^4}.\]
\bibliographystyle{abbrv}
\bibliography{biblio}

\end{document}